\documentclass{aims}
\usepackage{amsmath}
  \usepackage{paralist}
  \usepackage{color}
  \usepackage{graphics} 
  \usepackage{epsfig} 
\usepackage{graphicx}  
\usepackage{epstopdf}
 \usepackage[colorlinks=true]{hyperref}
\usepackage[all]{xy} 
\usepackage{amssymb} 
\hypersetup{urlcolor=blue, citecolor=red}
\usepackage{multicol}
\usepackage[active]{srcltx}
\usepackage{amsthm}
\usepackage{amsmath}
\usepackage{wasysym}
\def\currentvolume{X}
 \def\currentissue{X}
  \def\currentyear{200X}
   \def\currentmonth{XX}
    \def\ppages{X--XX}
     \def\DOI{10.3934/xx.xx.xx.xx}

\newtheorem{theorem}{Theorem}[section]

\newtheorem{lemma}[theorem]{Lemma}

\newtheorem{proposition}{Proposition}

\theoremstyle{definition}
\newtheorem{definition}[theorem]{Definition}
\newtheorem{remark}{Remark}

\def\R{\mathbb{R}}

\newcommand{\lp}{\left(}
\newcommand{\rp}{\right)}
\newcommand{\lc}{\left\{}
\newcommand{\rc}{\right\}}
\newcommand{\der}{\partial}

\newcommand{\bra}{\langle}
\newcommand{\ket}{\rangle}
\newcommand{\ra}{\rightarrow}

\def\se{\mathfrak{se}(2)}
\newcommand{\lsb}{\left[}

\newcommand{\rsb}{\right]}

\newcommand{\scp}[2]{{\left\langle {#1}\, , \, {#2}\right\rangle}}

\newcommand{\wlp}{\frac{\partial\mathcal{L}}{\partial p^{(i)}}}
\newcommand{\wphip}{\frac{\partial\mathcal{\chi}^{\alpha}}{\partial p^{(i)}}}
\newcommand{\wls}{\frac{\partial\mathcal{L}}{\partial \sigma^{(i)}}}
\newcommand{\wphis}{\frac{\partial\mathcal{\chi}^{\alpha}}{\partial\sigma^{(i)}}}

\newcommand{\DD}[2]{\frac{D #1}{D #2}}

\newcommand{\N}{\mathbb{N}}      
\newcommand{\F}{\mathbb{F}}

\newcommand{\Flder}{\rightarrow}

\usepackage{amssymb}

\newcommand{\proa}{A^*G \mbox{$\;$}_{\tau^*} \kern-3pt\times_\alpha
G \mbox{$\;$}_\beta \kern-3pt\times_{\tau^*} A^*G}

\newcommand{\ad}{\mbox{ad}}

\newcommand{\al}{\mathfrak{g}}
\newcommand{\dal}{\mathfrak{g}^{*}}

\newcommand{\cone}{\mathcal{A}}
\newcommand{\dcone}{\mathcal{A}_d}
\newcommand{\ti}{z}

\title[Discretization of the higher-order Lagrange-Poincar\'e equations]{The variational discretization of the constrained higher-order Lagrange-Poincar\'e equations}
\author[Anthony Bloch, Leonardo Colombo and Fernando Jim\'enez]{}

\subjclass{Primary: 37K05 ; Secondary: 37J15; 37N05; 49M25; 49S05; 49J15.}
 \keywords{Variational integrators, discrete mechanical systems, Lagrange-Poincar\'e equations, geo\-me\-tric integration, discrete variational calculus, ordinary differential equations, control of mechanical systems, reduction by symmetries.}

 \email{abloch@umich.edu}
 \email{leo.colombo@icmat.es.edu}
 \email{fernando.jimenez@eng.ox.ac.uk}
\thanks{A. Bloch was supported by NSF grant DMS-1613819 and AFOSR grant
FA 9550-18-0028. L. Colombo was supported by MINECO (Spain) grant MTM2016-76072-P. F. Jim\' enez was supported by the EPSRC project: `'Fractional Variational Integration and Optimal Control"; ref: EP/P020402/1.
We thank the reviewers for their valuable comments, that have helped to improve this work.
}

\begin{document}
\maketitle

\centerline{\scshape Anthony Bloch}
\medskip
{\footnotesize
 \centerline{Department of Mathematics, University of Michigan}
   \centerline{530 Church Street}
   \centerline{ Ann Arbor, Michigan, 48109, USA}
} 

\medskip

\centerline{\scshape Leonardo Colombo}

{\footnotesize
 \centerline{Instituto de Ciencias Matem\'aticas}
 \centerline{Consejo Superior de Investigaciones Cient\'ificas}
   \centerline{Calle Nicol\'as Cabrera 13-15, Campus UAM}
   \centerline{Madrid, 28049}
} 

\medskip

\centerline{\scshape Fernando Jim\'enez}
\medskip
{\footnotesize
 \centerline{ Department of Engineering Science}\centerline{University of Oxford}
   \centerline{Parks Road, Oxford, OX1 3PJ, United Kindom.}
}

\bigskip

 \centerline{(Communicated by the associate editor name)}

\begin{abstract}
 In this paper we investigate a variational discretization for the class of mechanical systems in presence of symmetries described by the action of a Lie group 
which reduces the phase space to a  (non-trivial) principal bundle. By introducing a discrete connection we are able to obtain the discrete constrained higher-order Lagrange-Poincar\'e equations. These  equations describe the dynamics of a constrained Lagrangian system when the Lagrangian function and the constraints depend on higher-order derivatives such as the acceleration, jerk or jounces. The equations, under some mild regularity conditions, determine a well defined (local) flow which can be used to define  a numerical scheme to integrate the constrained higher-order Lagrange-Poincar\'e equations. 
 
 Optimal control problems for underactuated mechanical systems can be viewed as higher-order constrained variational problems. We study how a variational discretization can be used in the construction of variational integrators for optimal control of underactuated mechanical systems where control inputs  act soley 
on the base manifold  of a principal bundle (the shape space). Examples include the energy minimum control of an electron in a magnetic field and two coupled rigid bodies attached at a common center of mass.

\end{abstract}

\section{Introduction}
Reduction theory is one of the fundamental tools in the study of mechanical systems
with symmetries and it essentially concerns the removal of certain variables by using the symmetries of the system and the associated
conservation laws. Such symmetries arise when one has a  Lagrangian which is  invariant
under a Lie group action $G$, i.e. if the Lagrangian function is invariant under the tangent lift of the action of the Lie group on the configuration manifold $Q$. If we denote by $\Phi_g:Q\Flder Q$ this (left-) action, for $g\in G$ then the invariance condition under the tangent lift action is expressed by $L\circ T\Phi_g=L$. If such an invariance property holds when the action $\Phi_g$ is given by left translations on the group $G$, that is, $\Phi_g=L_g$ where $L_g:G\to G$ is given by $L_g(h)=gh$ we say that the Lagrangian $L$ is $G$-invariant. For
a symmetric mechanical system, reduction by symmetries eliminates the directions
along the group variables and thus provides a system with fewer degrees of freedom. 

If the (finite-dimensional) differentiable
manifold $Q$ has local coordinates $(q^i)$, $1\leq i\leq$ dim$\,Q$ and we denote by $TQ$ its
tangent bundle with induced local coordinates $(q^i, \dot{q}^i)$, given a Lagrangian function $L:TQ\rightarrow \R$, its Euler--Lagrange
equations are
\begin{equation}\label{qwer}
\frac{d}{dt}\left(\frac{\partial L}{\partial\dot
q^i}\right)-\frac{\partial L}{\partial q^i}=0, \quad 1\leq i\leq \mbox{dim}\,Q.
\end{equation} As is well-known, when $Q$ is the configuration manifold  of a mechanical system, equations \eqref{qwer} determine its dynamics.

A paradigmatic example of reduction is the derivation of the Euler-Poincar\'e equations from the Euler-Lagrange equations  \eqref{qwer} when the configuration manifold is a Lie group, i.e. $Q=G.$ Assuming that the Lagrangian $L:TG\to\mathbb{R}$ is left invariant under the action of $G$ it is possible to reduce the system by introducing the body fixed velocity $\xi\in\mathfrak{g}$  and the reduced Lagrangian $\ell:TG/G\simeq\mathfrak{g}\to\mathbb{R}$, provided by the invariance condition $\ell(\xi)=L(g^{-1}g,g^{-1}\dot{g})=L(e,\xi)$. The dynamics of the reduced Lagrangian is governed by the \textit{Euler--Poincar\'e} equations (see \cite{Bloch} and \cite{BlKrMaRa} for instance) and given by the system of first order ordinary differential equations
\begin{equation}\label{Euler-Poincare-Eq}
\frac{d}{dt}\lp\frac{\partial\ell}{\partial\xi}\rp=\ad^{*}_{\xi}\lp\frac{\partial\ell}{\partial\xi}\rp.
\end{equation}
This system, together with the reconstruction equation $\xi(t)=g^{-1}(t)\dot{g}(t)$, is equivalent to the Euler-Lagrange equations on $G$, which are given by
\[
\frac{d}{dt}\left(\frac{\partial L}{\partial\dot
g}\right)=\frac{\partial L}{\partial g}\Rightarrow \left\{
\begin{array}{l}
\displaystyle\dot g=g\xi,\\\\
\displaystyle\frac{d}{dt}\lp\frac{\partial\ell}{\partial\xi}\rp=\ad^{*}_{\xi}\lp\frac{\partial\ell}{\partial\xi}\rp.
\end{array}
\right.
\] 

Reduction theory for mechanical systems with symmetries can be also developed by
using a variational principle formulated on a
principal bundle $\pi:Q \to Q/G$, where the principal connection $\mathcal{A}$ is introduced on $Q$ \cite{CeMaRa} (see Definition \ref{def:connection}). The connection yields the bundle isomorphism
$\alpha_{\mathcal{A}}^{(1)}:(TQ)/G \rightarrow T (Q/G) \oplus_{Q/G} \widetilde{\mathfrak {g}},
$
\[
\alpha_{\mathcal{A}}^{(1)}(\lsb v_q\rsb):=\left( T\pi\lp v_q\rp, [ q, \mathcal{A} (v_q)]_{\mathfrak{g}}\right),
\] (see equation \eqref{IsomorphismCont}) where the bracket is the standard Lie bracket on the Lie algebra $\mathfrak{g}$ and $\widetilde{\mathfrak{g}}:=\hbox{Ad} Q$ is the adjoint bundle
$\hbox{Ad} Q:= (Q\times\mathfrak{g})/G$. A curve $q(t)\in Q$ induces the two curves
$p(t):=\pi(q(t))\in Q/G$ and $\sigma(t) = [ q(t), \mathcal{A} (\dot q(t))]_{\mathfrak{g}}\in\widetilde{\mathfrak{g}} .$

Variational Lagrangian reduction \cite{CeMaRa} states that the
Euler-Lagrange equations on $Q$ with a $G$-invariant Lagrangian $L$
are equivalent to the Lagrange-Poincar\'e equations on $TQ/G \cong T
(Q/G) \oplus_{Q/G}  \widetilde{\mathfrak{g}}$ with reduced
Lagrangian $\mathcal{L}:T(Q/G) \oplus_{Q/G}  \widetilde{\mathfrak{g}}\Flder\R$. The Lagrange-Poincar\'e equations read
\begin{equation}\label{Intro-LagPoincareGeneral}
\left\{
\begin{array}{l}
\displaystyle\vspace{0.2cm}\DD{}{t}\frac{\partial \mathcal{L}}{\partial\sigma} - \mbox{ad}^*_{\sigma}\frac{\partial \mathcal{L}}{\partial\sigma}=0
,\\
\displaystyle\frac{\partial \mathcal{L}}{\partial p} - \DD{}{t}\frac{\partial \mathcal{L}}{\partial\dot{p}} = \scp{\frac{\partial \mathcal{L}}{\partial\sigma}}{i_{\dot{p}}\widetilde{\mathcal{B}}},
\end{array}
\right.
\end{equation} where $\widetilde{\mathcal{B}}$ is the reduced curvature form associated
to the principal connection $\mathcal{A}$ and $D/Dt$ denotes the 
covariant derivative in the associated bundle (see Definition \ref{curvature}).

The derivation of  variational integrators for \eqref{qwer} and \eqref{Euler-Poincare-Eq} from the discretization of variational principles has received a lot attention from the Dynamical Systems Geometric Mechanics community in the recent years \cite{MMM}, \cite{MMS}, \cite{MaWe}, \cite{MPS}, \cite{MPS2}, \cite{Mose} (and in particular for optimal control of mechanical systems \cite{Benito},  \cite{BlCrNoSa}, \cite{BlochLeok}, \cite{BHM}, \cite{Cedric1}, \cite{Cedric2}, \cite{CJdD}, \cite{CMdD}, \cite{Sigrid}, \cite{Leok2}, \cite{Sina}). The preservation of the symplectic form and momentum map are important properties which guarantee the competitive qualitative and quantitative behavior of the proposed
methods and mimic the corresponding properties of the continuous problem. That is, these methods allow substantially more accurate simulations at lower cost for higher-order problems with constraints.  Moreover, if the system is subject to constraints, then, under a regularity condition, it can be shown  that the system also preserves a symplectic form or a Poisson structure in the reduced case (\cite{CoMdDZu2013} and \cite{CMdD} for instance).

The construction of variational integrators for mechanical systems where the configuration space is a principal bundle has been studied in the geometric framework of Lie groupoids \cite{MMM} and as a motivation for the construction of a discrete time connection form \cite{Leok1}, \cite{FZ2}. This line of research has  been further developed in the last decade by T. Lee, M. Leok and H. McClamroch \cite{melvinbook}. We  focus on systems whose phase space is of  higher-order, i.e. $T^{(k)}Q$ \cite{IFAC}, \cite{GHMRV10}, \cite{GHMRV12}, and moreover is invariant under the action of symmetries. The Euler-Lagrange and Lagrange-Poincar\'e equations  for these systems were introduced by F. Gay-Balmaz, D. Holm and T. Ratiu in \cite{GHR2011}.  In this work, we aim to develop their discrete analogue for non-trivial principal bundles and its extension to constrained systems (where the constraints will be as well of the higher-order type).  With this in mind we employ the discrete Hamilton's principle by introducing a discrete connection  and using  Lagrange multipliers in order to obtain discrete paths that approximately satisfy the dynamics and the constraints. 
 As examples, we will illustrate our theory by applying the obtained discrete equations to the problem of energy minimum control of an electron in a magnetic field and two coupled rigid bodies attached at a common center of mass.

The structure of the work is as follows: Section $2$ introduces preliminaries on geometric mechanics, Lagrange-Poincar\'e equations, higher-order tangent bundles and the derivation of the constrained higher-order Lagrange-Poincar\'e equations (Theorem \ref{LPtheorem}). Section $3$ starts by introducing discrete mechanics and the notion of discrete connection. Next, we study the variational discretization of the constrained higher-order Lagrange-Poincar\'e equations to obtain a discrete time flow that integrates the continuous time constrained higher-order Lagrange-Poincar\'e equations. Moreover we provide sufficient regularity conditions for the discrete flow to exist. We proceed by treating the second-order case (the discrete constrained Lagrange-Poincar\'e equations are given in Theorem \ref{TheoremDisc1} and the regularity conditions in Proposition \ref{proposition1}) as an illustration of our approach. Then we 
carry out the  full higher-order case (the equations are given in Theorem \ref{HOTheorem}, while the regularity conditions are in Proposition \ref{proposition2}). Finally, in Section $4$, we apply the discrete equations to underactuated mechanical systems in two examples of optimal control, showing that they give rise to a meaningful discretization of the continuous systems.

\section{Constrained higher-order Lagrange-Poincar\'e equations}
In this section we introduce some preliminaries about geometric mechanics on Lie groups, Lagrange-Poincar\'e reduction, higher order tangent bundles and we study the constrained variational principle for higher-order mechanical systems on principal bundles.

\subsection{Mechanics on Lie groups and Euler-Poincar\'e equations}

\begin{definition}A \textit{Lie group} is a smooth manifold $G$ that is a group and for which the operations of multiplication $(g,h)\mapsto gh$ for $g,h\in G$ and inversion, $g\mapsto g^{-1}$, are smooth.
\end{definition}

\begin{definition}
A \textit{symmetry} of a function $F:G\to\mathbb{R}$ is a map $\phi:G\to G$ such that $F\circ\phi=F$. In such a case $F$ is said to be a $G$-invariant function under $\phi$.
\end{definition}
\begin{definition}Let $G$ be a Lie group with identity element $e\in G$. A \textit{left-action} of $G$ on a manifold $Q$ is a smooth mapping $\Phi:G\times Q\to Q$ such that $\Phi(e,q)=q$ $\forall q\in Q$, $\Phi(g,\Phi(h,q))=\Phi(gh,q)$ $\forall g,h\in G, q\in Q$ and for every $g\in G$, $\Phi_g:Q\to Q$ defined by $\Phi_{g}(q):=\Phi(g,q)$ is a diffeomorphism. 

$\Phi:G\times Q\to Q$ is a right-action if it satisfies the same conditions as for a left action
except that $\Phi(g,\Phi(h,q))=\Phi(hg,q)$ $\forall g,h\in G, q\in Q$. \end{definition}

We often use the notation $gq:=\Phi_{g}(q)=\Phi(g,q)$ and say that $g$ acts on $q$. All actions of Lie groups will be assumed to be smooth.

Let $G$ be a finite dimensional Lie group and let $\mathfrak{g}$ denote the Lie algebra associated to $G$ defined as $\mathfrak{g}:=T_{e}G$, the tangent space at the identity $e\in G$.  Let $L_{g}:G\to G$ be the left translation of the element $g\in G$ given by $L_{g}(h)=gh$ for $h\in G$. Similarly, $R_g$ denotes the right translation of the element $g\in G$ given by $R_{g}(h)=hg$ for $h\in G$. $L_g$ and $R_g$ are diffeomorphisms on $G$ and a left-action (respectively right-action) from $G$ to $G$ \cite{Holmbook}. Their tangent maps  (i.e, the linearization or tangent lift) are denoted by $T_{h}L_{g}:T_{h}G\to T_{gh}G$ and $T_{h}R_{g}:T_{h}G\to T_{hg}G$, respectively. Similarly, the cotangent maps (cotangent lift) are denoted by $T_{h}^{*}L_{g}:T^{*}_{h}G\to T^{*}_{gh}G$ and $T_{h}^{*}R_{g}:T^{*}_{h}G\to T^{*}_{hg}G$, respectively. It is well known that the tangent and cotangent lifts are actions (see \cite{Holmbook}, Chapter $6$).

Let $\Phi_g:Q\to Q$ for any $g\in G$ be a left action on $G$; a function $f:Q\to\mathbb{R}$ is said to be \textit{invariant} under the action $\Phi_g$, if $f\circ\Phi_g=f$, for any $g\in G$ (that is, $\Phi_g$ is a symmetry of $f$). The \textit{Adjoint action}, denoted $\hbox{Ad}_{g}:\mathfrak{g}\to\mathfrak{g}$ is defined by $\hbox{Ad}_{g}\chi:=g\chi g^{-1}$ where $\chi\in\mathfrak{g}$. Note that this action represents a change of basis on the Lie algebra.

If we assume that the Lagrangian $L\colon TG\to\mathbb{R}$ is $G$-invariant under the tangent lift of left translations, that is $L\circ T_{g}L_{g^{-1}}=L$ for all $g\in G$, then it is possible to obtain a reduced Lagrangian $\ell\colon\mathfrak{g}\to\mathbb{R}$, where $$\ell(\xi) = L(g^{-1}g,T_{g}L_{g^{-1}}(\dot{g}))= L(e,\xi).$$The reduced Euler--Lagrange equations, that is, the Euler--Poincar{\'e} equations (see, e.g., \cite{Bloch}, \cite{Holmbook}), are given by the system of $n$ first order ODE's \begin{align}
\frac{d}{dt}\frac{\partial\ell}{\partial\xi} = \ad^{*}_{\xi}\frac{\partial\ell}{\partial\xi}.\label{eq_ep_intro}
\end{align} where $\ad^{*}:\mathfrak{g}\times\mathfrak{g}^{*}\to\mathfrak{g}^{*}$, $(\xi,\mu)\mapsto\ad^{*}_{\xi}\mu$ is the \textit{co-adjoint operator} defined by $\langle\ad_{\xi}^{*}\mu,\eta\rangle=\langle\mu,\ad_{\xi}\eta\rangle$ for all $\eta\in\mathfrak{g}$ with $\ad:\mathfrak{g}\times\mathfrak{g}\to\mathfrak{g}$ the \textit{adjoint operator} given by $\ad_{\xi}\eta:=[\xi,\eta]$, where $[\cdot,\cdot]$ denotes the Lie bracket of vector fields on the Lie algebra $\mathfrak{g}$, and where $\langle\cdot,\cdot\rangle:\mathfrak{g}^{*}\times\mathfrak{g}\to\mathbb{R}$ denotes the so-called \textit{natural pairing} between vectors and co-vectors defined by $\langle\alpha,\beta\rangle:=\alpha\cdot\beta$ for $\alpha\in\mathfrak{g}^{*}$ and $\beta\in\mathfrak{g}$ where $\alpha$ is understood as a row vector and $\beta$ a column vector. For matrix Lie algebras $\langle\alpha,\beta\rangle=\alpha^{T}\beta$ (see \cite{Holmbook}, Section $2.3$ pp.$72$ for details).

Using this pairing between vectors and co-vectors one can write a useful relation between the tangent and cotangent lifts \begin{equation}\label{relation-cltl}\langle\alpha,T_{h}L_{g}(\beta)\rangle=\langle T^{*}_{h}L_{g}(\alpha),\beta\rangle\end{equation} for $g,h\in G$, $\alpha\in\mathfrak{g}^{*}$ and $\beta\in\mathfrak{g}$.

The Euler--Poincar{\'e} equations together with the reconstruction equation $\xi = T_{g}L_{g^{-1}}(\dot{g})$ are equivalent to the Euler--Lagrange equations on $G$.


\subsection{Geometry of principal bundles}

In this subsection we recall the basic tools for analysis of  the geometry of principal bundles that are useful in this paper (for more details see \cite{CeMaRa} and references therein).

\begin{definition}\label{def:connection}
Let $G$ be a Lie group and $\mathfrak{g}$ its Lie algebra. Given a free and proper left Lie group action $\Phi:G\times Q\to Q$, one can consider the principal bundle  $\pi:Q \to Q/G$. A \textit{connection} $\mathcal{A}$ on the principal bundle $\pi$ is a one-form on $Q$ taking values on $\mathfrak{g}$, such that $\mathcal{A}(\xi_{Q}(q))=\xi,$ for all  $\xi\in\mathfrak{g}, q\in Q$ and {\rm $\Phi^{*}_{g}\mathcal{A}=\mbox{Ad}_{g}\mathcal{A}$ where $\xi_{Q}$} is the infinitesimal generator associated with $\xi$ defined as $\xi_{Q}(q):=\frac{d}{dt}\Big{|}_{t=0}q\cdot\exp(t\xi).$
\end{definition}
The associated bundle $N$ with standard fiber $M$ (a smooth manifold), is defined as
\begin{equation}\label{AssBun}
N=Q\times_{G}M=(Q\times M)/G,
\end{equation}
where the action of $G$ on $(Q\times M)$ is diagonal, i.e. given by $g(q,m)=(gq,gm)$ for $q\in Q$ and $m\in M$. The orbit of $(q,m)$ is denoted $[q,m]_G$ or simply $[q,m]$. The projection $\pi_N:N\Flder Q/G$ is given by $\pi_N([q,m]_G)=\pi(q)$ and it is a surjective submersion.
The {\em adjoint bundle} is the associated vector bundle with $M=\mathfrak{g}$ under the adjoint action by the inverse element $g^{-1}\in G$, $\xi \mapsto$ Ad$_{g^{-1}}\xi$, and is denoted 
\begin{equation}\label{AdjointBundle}
\mbox{Ad}Q := Q \times_G\mathfrak{g}.
\end{equation}  
We will usually employ the short-hand notation $\widetilde{\mathfrak{g}}:=$Ad$Q$. The  orbits in this case are denoted $[q,\eta]_{\al}$ for $q\in Q$ and $\eta\in\al$ . Ad$Q$ is a Lie algebra bundle, that is, each fibre is a Lie algebra with the Lie bracket defined by
\[
\left[ [ q, \xi ]_{\mathfrak{g}}, [q, \eta]_{\mathfrak{g}}\right] = \left[ q, [\xi, \eta] \right]_{\mathfrak{g}}.
\]

Reduction theory for mechanical systems with symmetries can be performed by a variational principle formulated on a principal bundle $\pi:Q \to Q/G$, with fixed
principal connection $\mathcal{A}$ on $Q$ (see \cite{CeMaRa}). In other words, the reduced Lagrangian will be defined on the reduced space $TQ/G$, say $\mathcal{L}:TQ/G\Flder\R$. The bundle isomorphism  $\alpha_{\mathcal{A}}^{(1)}:TQ/G \rightarrow T (Q/G) \times_{Q/G} \widetilde{\mathfrak {g}}$, provided by the connection,  will facilitate the study of the suitable variations. It is defined by
\begin{equation}\label{IsomorphismCont}
\alpha_{\mathcal{A}}^{(1)}(\lsb v_q\rsb):=\left( T\pi\lp v_q\rp, [ q, \mathcal{A} (v_q)]_{\mathfrak{g}}\right),
\end{equation}
where the bracket is the standard Lie bracket on the Lie algebra $\mathfrak{g}$, $v_q\in T_{q}Q$ and $[v_q]\in (T_{[q]_{G}}Q)/G$ with $[q]_{G}\in Q/G$. A curve $q(t)\subset Q$ induces the two curves
$p(t):=\pi(q(t))\subset Q/G$ and $\sigma(t) := [ q(t), \mathcal{A} ((q(t),\dot q(t)))]_{\mathfrak{g}}\subset\widetilde{\mathfrak{g}},$ where we denote by $(q(t),\dot q(t))$ the local coordinates of $v_{q(t)}\in T_{q(t)}Q$ at each $t$.

\begin{definition}\label{curvature}The connection $\mathcal{A}$ also allows to define the curvature form $\mathcal{B}$, a $2$-form on $Q$ taking values on $\mathfrak{g}$, determined by 
$$\mathcal{B}( v_q, u_q):=\mathbf{d}\mathcal{A}(v_q, u_q)-[\mathcal{A}(v_q),\mathcal{A}(u_q)]_{\mathfrak{g}}\in\al,$$ where $u _q , v_q $ are arbitrary vectors in $T_qQ$ such that $T_q\pi(u_q)=u_p $ and $T_q \pi ( v _q )= v _p $, with $p=\pi(q)$. The curvature form $\mathcal{B}$ induces a  $\widetilde{\mathfrak{g}}$-valued two-form $\widetilde {\mathcal{B}} $ on $Q/G$ defined by
\begin{equation}\label{reducecurvature}
\widetilde{ \mathcal{B} }(u_p , v_p )=\left[q, \mathcal{B} _q( u _q , v _q )\right]_{\mathfrak{g}}\in\widetilde{\al},\quad u_p , v_p\in T_p (Q/G),
\end{equation}
where $u_q,v_q$  and $u_p,v_p$ are related as above. The two-form $\widetilde{ \mathcal{B} }$ is called the \emph{reduced curvature form} (for more details see \cite{CeMaRa} and references therein).
\end{definition}

\subsubsection{The covariant derivative}It is well know that the covariant derivative on a vector bundle induces an associated covariant derivative on its dual bundle. In this work, as in \cite{CeMaRa} and \cite{GHR2011}, we use this fact to define the covariant derivative in the dual of the adjoint bundle. If $\tilde{\sigma}(t)$ is a curve on $\tilde{\mathfrak{g}}^{*}$ the covariant derivative of $\tilde{\sigma}(t)$ is defined in such a way that for some curve $\sigma(t)$ on $\tilde{\mathfrak{g}}$, both, $\tilde{\sigma}(t)$ and $\sigma(t)$ project onto the same curve $p(t)$ on $Q/G$. Then $$\frac{d}{dt}\langle\tilde{\sigma}(t),\sigma(t)\rangle=\Big{\langle}\frac{D\tilde{\sigma}(t)}{Dt},\sigma(t)\Big{\rangle}+\Big{\langle}\tilde{\sigma}(t),\frac{D\sigma(t)}{Dt}\Big{\rangle}.$$ In the same way one can define the covariant derivative on $T^{*}(Q/G)$ and therefore a covariant derivative on $T^{*}(Q/G)\times_{Q/G}\tilde{\mathfrak{g}}^{*}$ (see \cite{CeMaRa} Section $3$ for more details).
\subsection{Lagrange-Poincar\'e reduction}\label{LPsubsection}
Lagrangian reduction by stages (\cite{CeMaRa}, Theorem 3.4.1) states that the
Euler-Lagrange equations \eqref{qwer} with a $G$-invariant Lagrangian $L:TQ\Flder\R$
are equivalent to the Lagrange-Poincar\'e equations on $TQ/G \cong T
(Q/G) \times_{Q/G}  \widetilde{\mathfrak{g}}$ (under the isomorphism \eqref{IsomorphismCont}) with reduced
Lagrangian $\mathcal{L}:T
(Q/G) \times_{Q/G}  \widetilde{\mathfrak{g}}\Flder\R$. The Lagrange-Poincar\'e equations read
\begin{equation}\label{LagPoincareGeneral}
\left\{
\begin{array}{l}
\displaystyle\vspace{0.2cm}\DD{}{t}\frac{\partial \mathcal{L}}{\partial\sigma} - \mbox{ad}^*_{\sigma}\frac{\partial \mathcal{L}}{\partial\sigma}=0
,\\
\displaystyle\frac{\partial \mathcal{L}}{\partial p} - \DD{}{t}\frac{\partial \mathcal{L}}{\partial\dot{p}} = \scp{\frac{\partial \mathcal{L}}{\partial\sigma}}{i_{\dot{p}}\widetilde{\mathcal{B}}},
\end{array}
\right.
\end{equation}
where $\widetilde{\mathcal{B}}$ is the reduced curvature form defined in \eqref{reducecurvature} and $D/Dt$ denotes
the covariant derivative in the associated bundle. Note that we are employing coordinates $(p,\dot{p},\sigma)$ for $T
(Q/G) \times_{Q/G}  \widetilde{\mathfrak{g}}$. Moreover, $i_{\dot{p}}\widetilde{\mathcal{B}}$ denotes the $\widetilde{\mathfrak{g}}$-valued $1$-form on $Q/G$ defined by $i_{\dot{p}}\widetilde{B}(\cdot)=\widetilde{B}(\dot{p},\cdot)$.

Consider a local trivialization of the principal bundle $\pi:Q\to Q/G$, i.e. a trivial principal bundle $\pi_{U}:U\times G\to U$ where $U$ is an open subset of $Q/G$  with structure group $G$ acting on the second factor by left multiplication. Denote by $(p^{s})$, $s=1,\ldots,r=\hbox{dim}(Q)-\hbox{dim}(G)$ local coordinates on $U$ and define maps $e_b:U\to\mathfrak{g}$ satisfying that for each $p\in U$, $\{e_b\}$ is a basis of $\mathfrak{g}$, $b=1,\ldots,\hbox{dim}(G)$. We  choose  the  standard  connection  on $U$, that is, at a tangent vector $(p,g,\dot{p},\dot{g})\in T_{(p,g)}(U\times G)$ we have $\mathcal{A}(p,g,\dot{p},\dot{g})=$Ad$_{g}(A_{e}(p)\dot{p}+\xi)$ where $\xi=g^{-1}\dot{g}$, $e$ is the identity of $G$, and $A_{e}:U\to\mathfrak{g}$ is a $1$-form given by $A_{e}(p)\dot{p}=\mathcal{A}(p,e,\dot{p},0)$.

Denote by $\bar{e}_{b}$ a section of $\widetilde{\mathfrak{g}}$ given by $\bar{e}_{b}(p)=[p,e,e_b(p)]_{\mathfrak{g}}$, $\sigma=\sigma^{a}\bar{e}_{a}$, and $\displaystyle{\bar{p}_b=\frac{\partial\mathcal{L}}{\partial\sigma}(\bar{e}_{b})}$. With this notation the Lagrange-Poincar\'e equations \eqref{LagPoincareGeneral} read (see \cite{MaSc} and \cite{CeMaRa} Section 4.2 for details)
\begin{equation}\label{LagPoinLocal}
\begin{split}
\frac{d}{dt}\bar{p}_{b}&=\bar{p}_a(C_{db}^{a}\sigma^{d}-C_{db}^{a}A_{s}^{d}\dot{p}^{s}),\\
\frac{\partial\mathcal{L}}{\partial p^{s}}-\frac{d}{dt}\frac{\partial\mathcal{L}}{\partial \dot{p}^{s}}&=\frac{\partial\mathcal{L}}{\partial\sigma^{a}}(B_{l\,s}^{a}\dot{p}^{l}+C_{db}^{a}\sigma^{d}A_{s}^{b}),
\end{split} 
\end{equation}
where $C_{bd}^{a}$ are the structure constants of the Lie algebra of $\mathfrak{g}$, $B_{l\,s}^{a}$ are the coefficients of the curvature in the local trivialization and $A_{s}^{a}(p)$ are the coefficients of $A_{e}$ for given local coordinates $p^{s}$ in $U$ determined by $(A_{e}(p)\dot{p})^{a}e_{a}=A_{s}^{a}(p)\dot{p}^{s}e_a$, $A_{e}(p)\dot{p}=\mathcal{A}(p,e,\dot{p},0).$

\subsection{Higher-order tangent bundles}
It is
possible to introduce an equivalence relation on the set $C^{k}(\R,
Q)$ of $k$-differentiable curves from $\R$ to $Q$ (see \cite{LR1} for more details): By definition,
two given curves in $Q$, $\gamma_1(t)$ and $\gamma_2(t)$, where
$t\in I\subset\R$ ($0\in I$), have a contact of order $k$
at $q_0 = \gamma_1(0) = \gamma_2(0)$, if there is a local chart
$(U,\varphi)$ of $Q$ such that $q_0 \in U$ and
$$\frac{d^s}{dt^s}\left(\varphi \circ \gamma_1(t)\right){\big{|}}_{t=0} =
\frac{d^s}{dt^s} \left(\varphi
\circ\gamma_2(t)\right){\Big{|}}_{t=0}\; ,$$ for all $s = 0,...,k.$
This is a well defined equivalence relation on $C^{k}(\R,Q)$ and the
equivalence class of a  curve $\gamma$ will be denoted by $[\gamma
]_{q_0}^{(k)}.$ The set of equivalence classes will be denoted by
$T^{(k)}Q$ and it is not hard to show that it has the natural
structure of a differentiable manifold. Moreover, $ \tau_Q^k  :
T^{(k)} Q \rightarrow Q$ where $\tau_Q^k
\left([\gamma]_{q_0}^{(k)}\right) = \gamma(0)$, is a fiber bundle
called the \emph{tangent bundle of order $k$} (or {\it higher-order tangent bundle}) of $Q$. In the sequel we will employ HO as short for higher-order.

Given a differentiable function $f: Q\longrightarrow \R$ and $l \in
\{0,...,k\}$, its $l$-lift $f^{(l, k)}$ to $T^{(k)}Q$, $0\leq l\leq
k$, is the
differentiable
function defined as
\[
f^{(l, k)}([\gamma]^{(k)}_0)=\frac{d^l}{dt^l}
\left(f \circ \gamma(t)\right){\Big{|}}_{t=0}\; .
\]
Of course, these definitions can be applied to functions defined on
open sets of $Q$.

From a local chart $(q^i)$ on a neighborhood $U$ of $Q$, it is
possible to induce local coordinates
 $(q^{(0)i},q^{(1)i},\dots,q^{(k)i})$ on
$T^{(k)}U=(\tau_Q^k)^{-1}(U)$, where $q^{(s)i}=(q^i)^{(s,k)}$ if
$0\leq s\leq k$. Sometimes, we will use the standard  conventions,
$q^{(0)i}\equiv q^i$, $q^{(1)i}\equiv \dot{q}^i$, $q^{(2)i}\equiv
\ddot{q}^i$, etc.

\subsubsection{HO quotient space:}

A smooth map $f:M\to N$ induces a map $T^{(k)}f:T^{(k)}M\to T^{(k)}N$ given by \begin{equation}\label{liftfunction}T^{(k)}f([\gamma]_{q_0}^{(k)}):=[f\circ\gamma]_{f(q_0)}^{(k)}.\end{equation} The action of a Lie group $\Phi_g$ is lifted to an action
$\Phi^{(k)}_g:T^{(k)}Q\ra T^{(k)}Q$, given by
$$\Phi_{g}^{(k)}([\gamma]_{q_0}^{(k)}):=T^{(k)}\Phi_{g}([\gamma]_{q_0}^{(k)})=[\Phi_{g}\circ
\gamma]^{(k)}_{\Phi_{g}(q_0)}.$$ If $\Phi_g$ is free and proper, we
get a principal $G$-bundle $T^{(k)}Q\ra T^{(k)}Q/G,$ which is a fiber bundle over $Q/G.$ The class of an element
$[\gamma]_{q_0}^{(k)}\in T^{(k)}_{q_0}Q$ in the quotient is denoted
$\displaystyle{[[\gamma]_{q_0}^{(k)}]_{G}}.$  From \cite{CeMaRa} (see Lemma 2.3.4) we know that the covariant
derivative of a curve $\sigma(t)=[q(t),\xi(t)]_{\al}\subset\widetilde{\mathfrak{g}}$ relative
to a principal connection $\mathcal{A}$ is given by
\begin{equation*}\label{CeMaRa}
\frac{D}{Dt}\sigma(t)=[q(t),\dot{\xi}(t)-[\mathcal{A}(q(t),\dot{q}(t)),\xi(t)]]_{\al}.\end{equation*}
In the particular case when
$\sigma(t)=[q(t),\mathcal{A}(q(t),\dot{q}(t))]_{\al}$ we have

$$
\frac{D}{Dt}\sigma(t)=[q(t),\dot{\xi}(t)]_{\al}\hbox{ and }
\frac{D^{2}}{Dt^{2}}\sigma(t)=[q(t),\ddot{\xi}(t)-[\xi(t),\dot{\xi}(t)]]_{\al}.
$$ If we denote by $\xi_1(t)=\xi(t),$ $\xi_2(t)=\dot{\xi}(t),$
$\xi_{3}(t)=\ddot{\xi}(t)-[\xi(t),\xi_{2}(t)],...,\xi_{k}(t)=\dot{\xi}_{k-1}(t)-[\xi(t),\xi_{k-1}(t)],$ operating recursively
one obtains $$\frac{D^{k-1}}{Dt^{k-1}}\sigma(t)=[q(t),\xi_k(t)]_{\al},$$ where $\xi_k\in\mathfrak{g}$ (see \cite{GHR2011} for example).

Taking all these elements into account, the bundle
 isomorphism that generalizes $\alpha_{\mathcal{A}}^{(1)}$ \eqref{IsomorphismCont} to the HO case is given by $\alpha_{\mathcal{A}}^{(k)}:T^{(k)}Q/G\ra T^{(k)}(Q/G)\times_{Q/G} k\tilde{\mathfrak{g}}$:
\begin{equation}\label{IsomorphismContHO}
\alpha_{\mathcal{A}}^{(k)}([[q]_{q_0}^{(k)}]_{G})=\left(T^{(k)}\pi([q]_{q_0}^{(k)}),\sigma(0),\frac{D}{Dt}\Big{|}_{t=0}\sigma(t),\frac{D^2}{Dt^2}\Big{|}_{t=0}\sigma(t),\ldots,\frac{D^{k-1}}{Dt^{k-1}}\Big{|}_{t=0}\sigma(t)\right).
\end{equation}
Note that with some abuse of notation we are denoting the class $[[\gamma]_{q_0}^{(k)}]_{G}$  by $[[q]_{q_0}^{(k)}]_{G}$.
In the expression \eqref{IsomorphismContHO}, $\sigma(t):=[q(t),\mathcal{A}(q(t),\dot{q}(t))]_{\al},$ $q(t)$
is any curve representing $[q]_{q_0}^{(k)}\in T^{(k)}Q$ with $q(0)=q_0,$ and $k\tilde{\mathfrak{g}}:=\underbrace{\tilde{\mathfrak{g}}\times\tilde{\mathfrak{g}}\times\ldots\times\tilde{\mathfrak{g}}}_{k-copies}$ (see \cite{CeMaRa} and \cite{Leok1}).

 For further purposes, it will be useful to establish a local notation for the reduced variables. We follow \cite{GHR2011} in this respect:
\begin{equation}\label{LocalCoordHO}
\alpha_{\mathcal{A}}^{(k)}([[q]_{q_0}^{(k)}]_{G})=(p,\dot{p},\ddot{p},\ldots,p^{(k)},\sigma,\dot{\sigma},\ddot{\sigma},\ldots,\sigma^{(k-1)}),
\end{equation}
where $(p,\dot{p},\ddot{p},\ldots,p^{(k)})$ are local coordinates on
$T^{(k)}(Q/G)$ and the dots denote the time derivatives in a
local chart;
$\sigma,\dot{\sigma},\ddot{\sigma},\ldots,\sigma^{(k-1)}$ are
independent elements in $\tilde{\mathfrak{g}},$ where we employ the notation 
 $\sigma^{(l)}:=\frac{D^{l}}{Dt^{l}}\sigma(t)$ for the covariant derivative. 
 
We introduce the notation $M^{(k)}:=T^{(k)}(Q/G)\times_{Q/G}k\widetilde{\mathfrak{g}}\Flder\R$, $M:=M^{(1)}$, and $s^{(k,k-1)}$ to denote the elements $(p,\dot{p},\ddot{p},\ldots,p^{(k)},\sigma,\dot{\sigma},\ddot{\sigma},\ldots,\sigma^{(k-1)})\in M^{(k)}$,  $s^{(k-1)}_{\sigma}=(\sigma,\dot{\sigma},\ddot{\sigma},\ldots,\sigma^{(k-1)})\in k\widetilde{\mathfrak{g}}$ and $s_{p}^{(k)}=(p,\dot{p},\ddot{p},\ldots,p^{(k)})\in T^{(k)}(Q/G)$.

\subsection{Constrained Hamilton's principle}
We derive the constrained HO Lagrange-Poincar\'e equations using the variational principles studied in \cite{CeMaRa} and \cite{GHR2011} for first order systems and unconstrained HO systems respectively. 

The constraint $\phi^{\alpha}:T^{(k)}Q\to\mathbb{R}$ is said to be $G$-invariant if it is invariant under the $k$-order tangent lift of left translations, that is, $$\phi^{\alpha}\circ T^{(k)}\Phi_g([\gamma]_{q_0}^{(k)})=\phi^{(k)}([\gamma]_{q_0}^{(k)})$$where $\Phi_g$ is just the left translation of the Lie group $L_g:G\to G$, and $T^{(k)}\Phi_g$ as in \eqref{liftfunction}.

Let $L:T^{(k)}Q\ra\R,$ and $\phi^{\alpha}:T^{(k)}Q\ra\R$ be a $G$--invariant
HO Lagrangian and $G$--invariant HO (independent)
constraints, respectively, $\alpha=1,\ldots, m$. 
The $G$-invariance allows to induce the reduced Lagrangian $\mathcal{L}:T^{(k)}Q/G\Flder\R$ and reduced
constraints $\mathcal{\chi}^{\alpha}:T^{(k)}Q/G\Flder\R$. 

After fixing a connection $\mathcal{A}$ we can employ the isomorphism \eqref{IsomorphismContHO}. Then it is possible to write  the reduced
Lagrangian and the reduced constraints  $\mathcal{L}:M^{(k)}\ra\R$ and 
$\mathcal{\chi}^{\alpha}:M^{(k)}\ra\R$, and employ the local coordinates $s^{(k,k-1)}$ as in \eqref{LocalCoordHO}.

\begin{remark}
Note that if $Q$ is the Lie group $G$, the adjoint bundle is identified with
$\mathfrak{g}$ via the isomorphism
$\alpha^{k}_{\mathcal{A}}:T^{(k)}G/G\ra
k\tilde{\mathfrak{g}}\cong k\mathfrak{g}:$
$$\alpha_{\mathcal{A}}^{(k)}([[g]_{g_{0}}^{(k)}]_{G})=\left(g^{-1}(0)\dot{g}(0),\frac{d}{dt}\Big{|}_{t=0}\xi(t),\ldots,\frac{d^{k-1}}{dt^{k-1}}\Big{|}_{t=0}\xi(t)\right),$$
where $\xi(t)=g^{-1}(t)\dot{g}(t).$ If we choose $g_0=e,$ that is,
$[[g_0g]_{e}^{(k)}]_{G}=[[g]_{g_0}^{(k)}]_{G},$ we can define the reduced Lagrangian and
the reduced constraints given by $\mathcal{L}:k\mathfrak{g}\ra\R$ and
$\mathcal{\chi}^{\alpha}:k\mathfrak{g}\ra\R$ (see \cite{CeMaRa}).\hfill$\diamond$

\end{remark}

In order to establish the variational principle, we must derive the variations on $T^{(k)}(Q/G)\oplus k\tilde{\mathfrak{g}}$ induced by variations on $Q,$
i.e. $\displaystyle{\delta q(t)=\frac{d}{ds}\Big{|}_{s=0}q(t,s)}\in T_{q(t)}Q$ at each $t$. Consider  an arbitrary deformation $p(t,s)\oplus\sigma(t,s)$ with $p(t,0)\oplus\sigma(t,0)=p(t)\oplus\sigma(t)$, the corresponding covariant variation is 
$$\delta p(t)\oplus\delta\sigma(t):=\frac{\partial}{\partial s}\Big{|}_{s=0}p(t,s)\oplus\frac{D}{Ds}\Big{|}_{s=0}\sigma(t,s).$$

Since
$s^{(k)}_p=T^{(k)}\pi([q]_{q_0}^{(k)}):=[\pi\circ
q]_{p}^{(k)}$, the variations $\delta p$ of $p(t)$ are
arbitrary except at the extremes; that is, $\delta p^{(l)}(0)=\delta
p^{(l)}(T)=0$ for $l=1,\ldots,k-1;$ $t\in [0,T].$ Then, locally we have that \begin{equation}\label{VarP}
\delta s^{(k)}_{p}:=(\delta
p,\delta\dot{p},\ldots,\delta p^{(k)}).
\end{equation}

On the other hand,  the covariant variations
of $\sigma$ are given by,
\begin{align*}\delta\sigma(t)&=\frac{D}{Dt}[q(t),\mathcal{A}(q(t),\delta
q(t))]_{\al}+[q(t),\mathcal{B}(\delta q(t),\dot{q}(t))]_{\al}\\&+[q(t),[\mathcal{A}(q(t),\dot{q}(t)),\mathcal{A}(q(t),\delta
q(t))]]_{\al}.\end{align*} In general, (see \cite{GHR2011}) for $\widetilde{\mathcal{B}}$ the reduced curvature \eqref{reducecurvature}, it follows that \begin{equation}\label{Varsigma}
\delta\sigma^{(j)}(t)=\frac{D^i}{Dt^i}\delta\sigma(s,t)+\sum_{j=0}^{i-1}\frac{D^j}{Dt^{j}}[\widetilde{\mathcal{B}}(p)(\dot{p}(t),\delta p(t)),\sigma^{(i-1-j)}(t)].
\end{equation}
Note that in the expression above,  $[\cdot,\cdot]$  denotes the usual Lie algebra bracket in $\widetilde{\al}$.
 
Consider the augmented Lagrangian $\widetilde{\mathcal{L}}:M^{(k)}\times\R^m\Flder\R$ given by $$\widetilde{\mathcal{L}}(s^{(k,k-1)}, \lambda_{\alpha})=\mathcal{L}(s^{(k,k-1)})+\lambda_{\alpha}\mathcal{\chi}^{\alpha}(s^{(k,k-1)})$$ where $\lambda_{\alpha}=(\lambda_1,\ldots,\lambda_m)\in\R^{m}$.  

A curve on $\gamma(t)\in C^{\infty}(\R,M^{(k)}\times\R^m)$ is locally represented by 
$$\gamma(t)=(s^{(k,k-1)}(t),\lambda_{\alpha}(t)).$$
Constrained HO Lagrange-Poincar\'e equations are derived by considering the constrained variational principle for the action
$S:C^{\infty}(\R,M^{(k)}\times\R^m)\ra\R$
given by
$$\displaystyle{ S(\gamma)=\int_{0}^{T}\widetilde{\mathcal{L}}(\gamma(t))\,dt}$$ for variations $\delta s^{(k,k-1)}=(\delta p,\delta\dot{p},\ldots,\delta p^{(k)},\delta\sigma,\delta\dot{\sigma},\ldots,\delta\sigma^{(k-1)})$ such that
\begin{enumerate}
\item[(I)] $\delta p^{(j)}(0)=\delta p^{(j)}(T)=0$, for $j=1,\ldots,k-1$
\item[(II)] Variations $\delta\sigma^{(j)}$ are of the form \eqref{Varsigma}, for $j=0,\ldots,k-1$,
\item[(III)] $\delta\sigma=\frac{D}{Dt}\Xi+[\sigma,\Xi]+\widetilde{\mathcal{B}}(\dot{p},\delta
p)$ where $\Xi$ is an arbitrary curve in $\widetilde{\mathfrak{g}}$
with $\frac{D^j}{Dt^j}\Xi$ vanishing at the endpoints,
\end{enumerate}

\begin{theorem}\label{LPtheorem}
Let $L:T^{(k)}Q\ra\R$ be a $G$-invariant Lagrangian and
$\chi^{\alpha}:T^{(k)}Q\ra\R$  $G$-invariant constraints,
$\alpha=1,\ldots,m$. Consider the principal $G$-bundle $\pi:Q\ra
Q/G$ and let $\mathcal{A}$ be a  principal connection on $Q.$ Let
$\mathcal{L}:M^{(k)}\ra\R$
and
$\mathcal{\chi}^{\alpha}:M^{(k)}\ra\R$
be the reduced HO Lagrangian and the reduced HO
constraints, respectively, associated with $\mathcal{A}$.

The curve $\gamma(t)\in M^{(k)}$ satisfies $\delta S(\gamma)=0$ with respect to the variations $\delta s^{(k,k-1)}$ satisfying (I)-(III) if and only if $\gamma(t)$ is a solution of the
 constrained HO Lagrange-Poincar\'e equations given by
\begin{equation}\label{HOLagPoi}
\begin{split}
0&=\sum_{i=0}^{k}(-1)^{(i)}\frac{d^{(i)}}{dt^{(i)}}\left(\wlp+\lambda_{\alpha}\wphip\right)-\Big{\langle}\sum_{i=0}^{k-1}\left((-1)^{i}\frac{D^{(i)}}{Dt^{(i)}}\left(\wls+\lambda_{\alpha}\wphis\right)\right.\\
&-\left.\sum_{l=0}^{i-1}(-1)^{l}\mbox{ad}_{{\sigma}^{(i-1-l)}}^{*}\frac{D^{(l)}}{Dt^{(l)}}\left(\wls+\lambda_{\alpha}\wphis\right)\right);i_{\dot{p}}\widetilde{\mathcal{B}}\Big{\rangle},\\
0&=\left(\frac{D}{Dt}-\mbox{ad}_{\sigma}^{*}\right)\sum_{i=0}^{k-1}(-1)^{(i)}\frac{D^{(i)}}{Dt^{(i)}}\left(\wls+\lambda_{\alpha}\wphis\right),\\
0&=\mathcal{\chi}^{\alpha}(\gamma(t)),
\end{split}
\end{equation}
where $i_{\dot{p}}\widetilde{\mathcal{B}}$ denotes the {\rm Ad}$Q$-valued $1$-form on $Q/G$ defined by $i_{\dot{p}}\widetilde{B}(\cdot)=\widetilde{B}(\dot{p},\cdot)$, given in \eqref{reducecurvature}.
\end{theorem}
\begin{proof}
The proof follows in a straightforward way by replacing the Lagrangian in the proof of Theorem $4.1$ of \cite{GHR2011} by the extended Lagrangian $\widetilde{\mathcal{L}}$.
\end{proof}
\begin{remark}
If $Q$ is a Lie group $G$ then the  constrained HO Lagrange-Poincar\'e equations \eqref{HOLagPoi} become the  constrained HO Euler-Poincar\'e equations, 	i.e: 
{\rm\begin{equation}\label{HOEuPoi}
\left\{
\begin{array}{l}
\displaystyle 0=\left(\frac{d}{dt}-\mbox{ad}_{\sigma}^{*}\right)\sum_{l=0}^{k-1}(-1)^{l}\frac{d^l}{dt^l}\left(\frac{\partial \mathcal{L}}{\partial\sigma}+\lambda_{\alpha}\frac{\partial\mathcal{\chi}^{\alpha}}{\partial\sigma}\right),\\
\displaystyle 0=\chi^{\alpha}(\gamma(t)),
\end{array}
\right.
\end{equation}}
where $\gamma(t)=(s^{(k-1)}_{\sigma})(t),\lambda_{\alpha}(t))\in C^{\infty}(\R,k\mathfrak{g}\times\R^m)$ (see \cite{CMdDJGM} and \cite{GHMRV10}).\hfill$\diamond$ \end{remark}
\begin{remark}\label{remarklocal2}

In the examples, we will be interested in the case $k=2$. In that case, the second-order Lagrange-Poincar\'e equations are locally written as:
\begin{align*}
\frac{\partial\widetilde{\mathcal{L}}}{\partial p^{s}}-\frac{d}{dt}\left(\frac{\partial\widetilde{\mathcal{L}}}{\partial\dot{p}^{s}}\right)+\frac{d^{2}}{dt^2}\left(\frac{\partial\widetilde{\mathcal{L}}}{\partial\ddot{p}^{s}}\right)&=\left(\frac{d}{dt}\frac{\partial\widetilde{\mathcal{L}}}{\partial\dot{\sigma}^{a}}-\frac{\partial\widetilde{\mathcal{L}}}{\partial\sigma^{a}}\right)\left(B_{ls}^{a}\dot{p}^{l}+C_{db}^{a}A_s^{b}\sigma^{d}\right),\\
\frac{d^{2}}{dt^2}\left(\frac{\partial\widetilde{\mathcal{L}}}{\partial\dot{\sigma}^{b}}\right)-\frac{d}{dt}\left(\frac{\partial\widetilde{\mathcal{L}}}{\partial\sigma^{b}}\right)&=\left(\frac{d}{dt}\frac{\partial\widetilde{\mathcal{L}}}{\partial\dot{\sigma}^{a}}-\frac{\partial\widetilde{\mathcal{L}}}{\partial\sigma^{a}}\right)(C_{db}^{a}\sigma^{d}-C_{db}^{a}A_{s}^{d}\dot{p}^{s}),\\
\chi^{\alpha}(s^{(2,1)})&=0,
\end{align*} 
where $\widetilde{\mathcal{L}}(s^{(2,1)},\lambda_{\alpha})=\mathcal{L}(s^{(2,1)})+\lambda_{\alpha}\chi^{\alpha}(s^{(2,1)})$.

Note that these equations are the second-order constrained version of the local expression of the Lagrange-Poincar\'e equations derived by Marsden and Scheurle in \cite{MaSc}. \hfill$\diamond$ 
\end{remark}

\section{Discrete constrained higher-order Lagrange-Poincar\'e equations}

\subsection{Discrete mechanics and variational integrators}\label{DMVI}

Variational integrators are a class of geometric integrators which are determined by a discretization of a variational principle. As a consequence,  some of the main geometric properties of  continuous
system, such as symplecticity and momentum conservation, are present in these numerical methods (see
\cite{Hair},\cite{MaWe} and \cite{Mose} and references therein).
In the following we will summarize the main features of this type
of geometric
 integrator. 
 
A {\it discrete Lagrangian} is a map
$L_d:Q\times Q\rightarrow\R$, which may be considered as
an approximation of the action integral defined by a continuous  Lagrangian $L\colon TQ\to
\R$, that is
\begin{equation}\label{aprox}
L_d(q_0, q_1)\approx \int^{t_0+h}_{t_0} L(q(t), \dot{q}(t))\; dt,
\end{equation}
where $q(t)$ is a solution of the Euler-Lagrange equations \eqref{qwer} joining $q(t_0)=q_0$ and $q(t_0+h)=q_1$ for small enough $h>0$, where $h$ is viewed as the step size of the integrator.

 Define the {\it action sum} $S_d\colon Q^{N+1}\to
\R$  corresponding to the Lagrangian $L_d$ by
\begin{equation}\label{DiscAction}
{S_d}=\sum_{n=0}^{N-1}  L_d(q_{n}, q_{n+1}),
\end{equation}
where $q_n\in Q$ for $0\leq n\leq N$, $N$ is the number of discretization steps. The discrete variational
principle   states that the solutions of the discrete system
determined by $L_d$ must extremize the action sum given fixed
endpoints $q_0$ and $q_N$. By extremizing ${S_d}$ over $q_n$,
$1\leq n\leq N-1$, it is possible to derive the system of difference equations
\begin{equation}\label{discreteeq}
 D_1L_d( q_n, q_{n+1})+D_2L_d( q_{n-1}, q_{n})=0.
\end{equation}
These  equations are usually called the  {\it discrete
Euler--Lagrange equations}. If the
matrix $D_{12}L_d(q_n, q_{n+1})$ is regular, it is possible to
define  from (\ref{discreteeq}) a (local) discrete flow map $ F_{L_d}\colon Q\times
Q\to  Q\times Q$, by $F_{L_d}(q_{n-1}, q_n)=(q_n,
q_{n+1})$. We will refer to the $F_{L_d}$ flow, and also (with some abuse of notation) to the equations \eqref{discreteeq}, as a {\it variational integrator.}  Using the discrete Legendre transformations $\F^{\pm}L_d:Q\times Q\Flder T^*Q$ (which we assume regular), one can construct the discrete Hamiltonian flow $\tilde F_{L_d}:T^*Q\Flder T^*Q$ out of the discrete Lagrangian one, namely $\tilde F_{L_d}=\F^{\pm}L_d\circ F_{L_d}\circ(\F^{\pm}L_d)^{-1}$, see \cite{MaWe}.

The choice of discrete Lagrangian \eqref{aprox} is crucial in the discrete variational procedure, since it determines the order of local truncation error of the discrete flows with respect to the continuous ones. The optimal approximation is given by the so-called \textit{exact discrete Lagrangian} \cite{Hair,MaWe}, say
\[
L_d^E(q_0, q_1)=\int^{t_0+h}_{t_0} L(q(t), \dot{q}(t))\; dt,
\]
which provides the exact continuous flow in one time step $h$ via the discrete Euler-Lagrange equations \eqref{discreteeq}. Nevertheless, the choice of $L_d^E$ is not practical since it involves the analytic solution of the continuous Euler-Lagrange equations; thus we need to take approximations. It was proven in \cite{MaWe} and \cite{PatrickCuell} that, if $||L_d^E-L_d||\sim O(h^{r+1})$, for $r\in\N$, then the local truncation error of the variational integrator is of the same order, i.e. $||\tilde F_{L_d}(q_0,p_0)-(q(h),p(h))||\sim O(h^{r+1})$, where $(q,p)$ are the coordinates of $T^*Q$, we define by $(q(t),p(t))$ the continuous flow and we set $(q(0),p(0))=(q_0,p_0)$ (equivalent results can be established at a Lagrangian level). Furthermore, the symplecticity of $\tilde F_{L_d}$ ensures its stability in the long-term performance when $h\Flder 0$, as proven in \cite{Hair} (Backward Error Analysis). In other words, if $H:T^*Q\Flder\R$ determines the energy of the system, then for a $r$-th order consistent $\tilde F_{L_d}$, we have that $||H(\tilde F^N_{L_d}(q_0,p_0))-H(q_0,p_0)||\sim O(h^{r+1}).$

Since $L_d^E$ is normally not available, what one can pick is the order of the approximation of $L_d$. This is done by the interpolation of the continuous curves $q(t)$ and $\dot q(t)$ in the right hand side of \eqref{aprox}. First order interpolations lead to the well-known midpoint rule, leapfrog, RATTLE and St\"ormer-Verlet methods \cite{MaWe}. High-order interpolations lead to higher-order approximations of $L_d^E$ and consequently to higher-order variational integrators, see for instance \cite{Cedric2, Sina2} (note that with {\it high-order} we refer here to the local truncation error of the numerical methods).

In the following sections we are going to concentrate on discrete HO problems with symmetry, which are the main topic of this work. Our focus is on the discretization procedure and the mathematical tools that it involves, whereas the numerical behavior is planned to be explored in further works. However, regarding the local truncation error and stability (which will be still ensured thanks to the symplecticity of the HO variational integrators), discussed in the last paragraphs, the construction of $L_d^E$ for HO systems has been developed in \cite{CoFeMdD}. In this reference the reader can find some numerical tests and more details  on numerical aspects.


\subsection{The discrete connection}\label{Discretization}

The discretization of the reduced HO tangent bundle $T^{(k)}Q/G\cong M^{(k)}$ is based on the decomposition of the space $(Q\times Q)/G$ by means of the so-called discrete connection \cite{Leok1} (see also \cite{FZ2}):

\begin{equation}\label{DiscConnec}
\dcone:Q\times Q\Flder G,
\end{equation}
which is defined to account for a reasonable discretization of the properties of the continuous connection $\mathcal{A}$ and, moreover, is $G$-equivariant (see \cite{FZ2,Leok1, MMM} for more details). 

Important properties that characterize the discrete connection are \cite{Leok1}:
\begin{enumerate}
\item $\dcone(q_0,g\,q_0)=g$,
\item $\dcone(g\,q_0,h\,q_1)=h\dcone(q_0,q_1)g^{-1},$
\item Consider a local trivialization of the principal bundle $\pi:Q\to Q/G$; namely, for any open neighborhood $V\subset Q$ we have
\begin{equation}\label{TrivialChart}
V\cong U\times G,
\end{equation}
where $U\subset Q/G$. In other words, for any $U\subset Q/G$, $\pi^{-1}(U)\cong U\times G$. In such a case, we have: 
\begin{equation}\label{DiscConeTrivialChart}
\dcone((p_0,g_0),(p_1,g_1))=g_1\dcone((p_0,e),(p_1,e))g_0^{-1},
\end{equation}
where locally $\pi(q_0)=\pi((p_0,g_0))=p_0.$ This defines the local expression of ·$\dcone$, say $\dcone((p_0,e),(p_1,e))=A(p_0,p_1)\in G$, which according to \eqref{DiscConeTrivialChart} leads to
\begin{equation}\label{LocalDiscConne}
\dcone((p_0,g_0),(p_1,g_1))=g_1\,A(p_0,p_1)\,g_0^{-1}.
\end{equation}
In particular, if $Q=G$, this leads to $\dcone((e,e),(e,e))=e$ and consequently $\dcone(g_0,g_1)=g_1g^{-1}_0.$
\end{enumerate}
\begin{remark}
Sometimes (see \cite{MMM}) the discrete connection is defined as the application $\mathcal{A}_{d}:Q\times Q\to G$ satisfying the properties $(1)$ and $(2)$ listed above.\hfill$\diamond$
\end{remark}
In particular, given a discrete connection $\dcone$ the following isomorphism  between bundles is well-defined (see \cite{Leok1} for the proof):
\begin{equation}\label{IsomorphDisc}
\begin{split}
\alpha_{_{\dcone}}^{(1)}: (Q\times Q)/G &\,\,\Flder\,\,\, \lp (Q/G)\times(Q/G)\rp\oplus\widetilde G,\\
[q_0,q_1]_G&\,\,\mapsto (\pi(q_0),\pi(q_1))\oplus[q_0,\dcone(q_0,q_1)]_G,
\end{split}
\end{equation}
$[q_0,\dcone(q_0,q_1)]_G\in \widetilde G$, where we denote $\widetilde G:=(Q\times G)/G$  in analogy with the adjoint bundle $\tilde\al$.  We note that \eqref{IsomorphDisc} is the discrete counterpart of the isomorphism \eqref{IsomorphismCont}.
\begin{remark}
In the case $Q=G$, the isomorphism \eqref{IsomorphDisc} is given by $\dcone(g_0,g_1)=g_0^{-1}g_1$ in view of property (3), which leads to the usual Euler-Poincar\'e discrete reduction as in \cite{MPS}. \hfill$\diamond$ 
\end{remark}

We consider the following extension of \eqref{IsomorphDisc} in the case of HO tangent bundles, which is local for non-trivial bundles (see \cite{FZ2,Leok1} for more details):
\begin{equation}\label{HOSplit}
\begin{split}
\alpha_{\dcone}^{(k)}:\lc Q^{(k+1)}\rc/G&\,\,\,\longrightarrow (Q/G)^{(k+1)}\times_{_{Q/G}}\widetilde G^{(k)},\\
	[q_0,...,q_k]_G&\,\,\,\longmapsto (\pi(q_0),...,\pi(q_k))\times_{Q/G}\bigoplus_{n=0}^{k-1}[q_0,\dcone(q_n,q_{n+1})]_G
\end{split}
\end{equation}
where $Q^{(k+1)}$ denotes the Cartesian product of $(k+1)$-copies of $Q$, $(Q/G)^{(k+1)}$ denotes the Cartesian product of $(k+1)$-copies of $(Q/G)$ and $\widetilde G^{(k)}$ denotes the sum of $k$-copies of $\widetilde G$. Consequently, we consider $H^{(k+1,k)}:=(Q/G)^{(k+1)}\times_{Q/G}\,\widetilde G^{(k)}$ as the discretization of the space $M^{(k)}$, which is natural according to \cite{Benito,Leok1,MaWe}.

\subsection{Variational discretization of constrained HO Lagrange-Poincar\'e equations}
In the following, we aim to derive the variational discrete flow obtained from a discretizacion of $\mathcal{L}$ and $\chi^{\alpha}$. Therefore, we shall work in local coordinates, particularly in the local trivialization of the principal bundle \eqref{TrivialChart}. 

The first task consists of obtaining the variational discretization of equations \eqref{HOLagPoi}. For this, we must fix the discrete connection \eqref{DiscConnec} and the discrete isomorphism \eqref{HOSplit}. Next, we can induce through $\mathcal{A}_{d}$ and $\alpha_{\mathcal{A}_{d}}^{(k)}$ the discrete reduced HO Lagrangian and the discrete reduced HO constraints,
\begin{equation}\label{HOcase}
\mathcal{L}_d: H^{(k+1,k)}\Flder\R,\,\,\, \mbox{and}\,\,\,\chi_{d}^{\alpha}: H^{(k+1,k)}\Flder\R,
\end{equation}
for $\alpha=1,...,m$. 

For a clear exposition, first we develop the first-order case, i.e. $k=1$,  in the next subsection,  where the main objects employed in the HO  case shall be introduced.

\subsubsection{Discrete constrained Lagrange-Poincar\'e equations}

We shall consider the discrete reduced Lagrangian and discrete reduced constraints:
\begin{equation}\label{DiscreteSettingFO}
\mathcal{L}_d: H^{(2,1)} \Flder\R\quad\quad\, \mbox{and}\quad\quad\,\chi^{\alpha}_{d}: H^{(2,1)}\Flder\R.
\end{equation}
where $H^{(2,1)}=\lp(Q/G)\times (Q/G)\rp\times_{_{Q/G}}\,\widetilde G$ according to the notation introduced above. Moreover, we will employ the trivialization \eqref{TrivialChart} to fix a local representation of $\widetilde G$, and consequently of $[q_0,q_1]_G\in (Q\times Q)/G.$  Indeed, employing the discrete connection $\dcone$ and the isomorphism \eqref{IsomorphDisc}, we can make the following identification
\[
\lp\pi^{-1}(U)\times \pi^{-1}(U)\rp/G\cong \lp(U\times G)\times (U\times G)\rp/G\cong U\times U\times G.
\]
Moreover, one can prove that the map 
\begin{align}
\lp(U\times G)\times (U\times G)\rp/G &\,\,\,\longrightarrow U\times U\times G,\label{Bijecton}\\
[(p_0,g_0),(p_1,g_1)]&\,\,\,\longmapsto (p_0,p_1,\dcone((p_0,e),(p_1,g_0^{-1}g_1)))\nonumber\\&\qquad=(p_0,p_1,g_0^{-1}g_1A(p_0,p_1)),\nonumber
\end{align}
is a bijection (see \cite{MMM}), where 
\begin{equation}\label{ConneLocalG}
A:U\times U\Flder G
\end{equation}
is the local representation of the discrete connection as  established in \eqref{LocalDiscConne}. Therefore, in this  trivialization  we can define the local coordinates 
\begin{equation}\label{a_n}
a_n:=[q_n,q_{n+1}]_G=(p_n,p_{n+1},g_n^{-1}g_{n+1}A(p_n,p_{n+1})),
\end{equation} where $n$ is 0 or a positive integer. 

\begin{lemma}\label{lemma} The variations for an element  $a_n \in U\times U\times G$ defined in \eqref{a_n} are determined by 
\begin{equation}\label{VariationsW}
\begin{split}
\delta a_n:=\delta[q_n,q_{n+1}]_G=&(\delta p_n,\delta p_{n+1},-\eta_nW_nA(p_n,p_{n+1})+W_n\eta_{n+1}\,A(p_n,p_{n+1})\\
&+W_n\bra D_1A(p_n,p_{n+1}),\delta p_n\ket+W_n\bra D_2A(p_n,p_{n+1}),\delta p_{n+1}\ket),
\end{split}
\end{equation} where  $W_n=g_n^{-1}g_{n+1}\in G$ and 
$\delta g_n:=g_n\,\eta_n$, with $\eta_n\in\al$.
\end{lemma}
\begin{proof}
We observe that 
\begin{align}
&\delta[q_n,q_{n+1}]_G=(\delta p_n,\delta p_{n+1},-g_n^{-1}\delta g_n\,g_n^{-1}g_{n+1}A(p_n,p_{n+1})+g_n^{-1}\delta g_{n+1}\,A(p_n,p_{n+1})\nonumber\\
&\quad\quad\quad\quad\quad\quad+g_n^{-1}g_{n+1}\delta A(p_n,p_{n+1}))\label{Variations}\\
&\quad\quad\quad\quad\quad=(\delta p_n,\delta p_{n+1},-g_n^{-1}\delta g_n\,g_s^{-1}g_{n+1}A(p_n,p_{n+1})+g_n^{-1}\delta g_{n+1}\,A(p_n,p_{n+1})\nonumber\\
&\quad\quad\quad\quad\quad\quad+g_{n}^{-1}g_{n+1}\bra D_1A(p_n,p_{n+1}),\delta p_n\ket+g_n^{-1}g_{n+1}\bra D_2A(p_n,p_{n+1}),\delta p_{n+1}\ket)\nonumber
\end{align} where, for $i=1,2$; $D_iA(p_n,p_{n+1})$ is a one form on  $T^*_{p_j}U$ taking values on $T_{A(p_n,p_{n+1})}G$ for $j=n$ if $i=1$ and $j=n+1$ if $i=2$, according to \eqref{ConneLocalG}. Using $W_n=g_n^{-1}g_{n+1}\in G$ and 
$\delta g_n:=g_n\,\eta_n$, with $\eta_n\in\al$, then \eqref{Variations} can be rewritten as \eqref{VariationsW}.
\end{proof}
Given the grid $\{t_{n}=nh\mid n=0,\ldots,N\}$,
with $Nh=T$, define the discrete path space
$\mathcal{C}_{d}(U\times U\times G):=\{\gamma_{d}:\{t_{n}\}_{n=0}^{N}\ra U\times U\times G\}.$ This discrete path space is isomorphic to the
smooth product manifold which consists of $N+1$ copies of $U\times U\times G$ (which  is locally isomorphic to $N+1$ copies of $(\lp Q/G\times Q/G\rp\times_{Q/G}\times \widetilde G)$). The
discrete trajectory $\gamma_{d}\in\mathcal{C}_{d}(U\times U\times G)$ will be identified
with its image, i.e. $\gamma_{d}(t_n)=\{a_{n}\}_{n=0}^{N}$ where
$a_{n}=(p_n,p_{n+1},g_n^{-1}g_{n+1}A(p_n,p_{n+1}))$. Let us consider the reduced discrete Lagrangian $\mathcal{L}_d$ in \eqref{DiscreteSettingFO}. Define the discrete action sum, $\mathcal{S}_{d}:\mathcal{C}_{d}(U\times U\times G)\to\mathbb{R}$, by 
\begin{equation}\label{DiscAction}
\mathcal{S}_d(\gamma_{d})=\sum_{n=0}^{N-1}\mathcal{L}_d([q_n,q_{n+1}]_G)=\sum_{s=0}^{N-1}\mathcal{L}_d(p_n,p_{n+1},g_n^{-1}g_{n+1}A(p_n,p_{n+1}))
\end{equation}
where the  equality is established at a local level. From now on, we use the notation $A_n:=A(p_n,p_{n+1})$ and  $\displaystyle{\mathcal{S}_d(\gamma_{d})=\sum_{s=0}^{N-1}\mathcal{L}_d(a_n)}$.

The \textit{discrete constrained variational problem} associated with \eqref{DiscreteSettingFO}, consists of finding a discrete path $\gamma_d\in\mathcal{C}_{d}(U\times U\times G)$, given fixed boundary conditions, which extremizes the discrete action sum \eqref{DiscAction} subject to the discrete constraints $\chi^{\alpha}_{d}$. This constrained optimization problem is equivalent to studying the (unconstrained) optimization problem for the \textit{augmented Lagrangian} $\widetilde{\mathcal{L}}_{d}:H^{(2,1)}\times\R^{m}\Flder\R$ given by 
\begin{equation}\label{AugLag}
\widetilde{\mathcal{L}}_d([q_n,q_{n+1}]_G,\lambda_{\alpha}^{n})=\mathcal{L}_d([q_n,q_{n+1}]_G)+\lambda_{\alpha}^n\chi_d^{\alpha}([q_n,q_{n+1}]_G)
\end{equation}
 where 
$\lambda^n_{\alpha}=(\lambda_1^n,...,\lambda_m^n)\in\R^m$ are Lagrange multipliers. The associated action sum is given by 
\begin{equation}\label{AugmentedActionSum}
\mathcal{S}_d(\tilde\gamma_{d})=\sum_{n=0}^{N-1}\widetilde{\mathcal{L}}_d([q_n,q_{n+1}]_G,\lambda_{\alpha}^n)=\sum_{n=0}^{N-1}\widetilde{\mathcal{L}}_d(p_n,p_{n+1},g_n^{-1}g_{n+1}A(p_n,p_{n+1}),\lambda_{\alpha}^n),
\end{equation}
where again the equality is given at a local level and $\tilde\gamma_d\in\mathcal{C}_{d}(U\times U\times G\times\R^m):=\{\tilde\gamma_{d}:\{t_{n}\}_{n=0}^{N}\ra U\times U\times G\times \R^m\}$ is the discrete augmented path space. We establish the  result in the following theorem, where the discrete constrained Lagrange-Poincar\'e equations are obtained.

\begin{theorem}\label{TheoremDisc1}
A discrete sequence $\lc a_n,\lambda^n_{\alpha}\rc_{n=0}^N\in \mathcal{C}_{d}(U\times U\times G\times\R^m)$ is an extremum of the action sum \eqref{AugmentedActionSum}, with respect to  variations $\delta[q_n,q_{n+1}]_G$ set in \eqref{VariationsW} and endpoint conditions $\delta q_0=\delta q_N=0$ where $q_j=(p_j,g_j)$ (while the Lagrange multipliers are free), if  it is a solution of the discrete constrained Lagrange-Poincar\'e equations \eqref{DiscEqofMotionSimples}.
\end{theorem}

\begin{proof}
The poof will be divided into two parts. The first one consists on studying the variations of the action sum \eqref{DiscAction} associated with $\mathcal{L}_{d}$. After that, our result follows by the incorporation of the constraints and Lagrange multipliers  by considering  $\widetilde{\mathcal{L}}_{d}$ instead of $\mathcal{L}_{d}$ and \eqref{AugmentedActionSum} instead of \eqref{DiscAction}.

Taking variations on the discrete action sum \eqref{DiscAction} with $q_0=(p_0,g_0)$ and $q_N=(p_N,g_N)$ fixed, which in terms of variations implies $\delta p_0=\delta p_N=0$ and $\delta g_0=\delta g_N=0$, the latter leading to $\eta_0=\eta_N=0$, and using the Lemma \ref{lemma}, we obtain
\begin{align}
\delta \sum_{n=0}^{N-1}\mathcal{L}_d(p_n,p_{n+1},W_n\,A_n)=&\sum_{n=1}^{N-1}\bra D_1\mathcal{L}_d(a_n)+D_2\mathcal{L}_d(a_{n-1})\,,\,\delta p_n\ket \nonumber\\
&+\sum_{n=1}^{N-1}\bra T^*_{W_{n-1}}L_{W_{n-1}^{-1}}(T^{*}_{W_{n-1}A_{n-1}}R_{A_{n-1}^{-1}}D_3\mathcal{L}_d(a_{n-1})),\eta_n\ket\nonumber\\
&- \sum_{n=1}^{N-1}\bra T^*_{W_n}R_{W_{n}^{-1}}(T_{W_nA_n}^{*}R_{A_{n}^{-1}}D_3\mathcal{L}_d(a_{n}))\,,\,\eta_n\ket\label{Vari}\\
&+\sum_{n=1}^{N-1}\bra T_{W_nA_n}^{*}L_{W_n^{-1}}D_3\mathcal{L}_d(a_n),\bra D_1A(p_n,p_{n+1}),\delta p_n\ket\ket  \nonumber\\
&+\sum_{n=1}^{N-1}\bra T_{W_{n-1}A_{n-1}}^{*}L_{W_{n-1}^{-1}}D_3\mathcal{L}_d(a_{n-1}),\mathcal{D}\ket,\nonumber
\end{align}
 where $\mathcal{D}=\bra D_2A(p_{n-1},p_n),\delta p_n\ket$,  $D_i$ denotes the partial derivative with respect to the $i$-th variable, $R_{g},L_{g}:G\Flder G$ are the left and right translations by the group variables, while $T^*_hR_g:T^*_hG\Flder T^*_{hg}G$, $T^*_hL_g:T^*_hG\Flder T^*_{gh}G$ are their cotangent action. Therefore, $\delta\mathcal{S}_{d}=0$ for arbitrary variations implies
\begin{subequations}\label{DEoM}
\begin{align}
0&=D_1\mathcal{L}_d(a_n)+D_2\mathcal{L}_d(a_{n-1})+T^{*}\hat L_{(WA_1)}(n)D_3\mathcal{L}_d(a_n)\nonumber\\&+T^{*}\hat L_{(WA_2)}(n-1)D_3\mathcal{L}_d(a_{n-1}),\label{DEoMa}\\
0&= T^*_{W_{n-1}}L_{W_{n-1}^{-1}}(T^{*}_{W_{n-1}A_{n-1}}R_{A_{n-1}^{-1}}D_3\mathcal{L}_d(a_{n-1}))\nonumber\\&- T^*_{W_{n}}R_{W_{n}^{-1}}(T_{W_nA_{n}}^{*}R_{A_{n}^{-1}}D_3\mathcal{L}_d(a_{n})),\label{DEoMb}
\end{align}
\end{subequations}
for $n=1,...,N-1$, where we define locally the operator $T^{*}\hat L_{(WA_i)}$ by its action on $T^*G$. Namely, $T^{*}\hat L_{(WA_i)}(j):T_{gA_i}^*G\Flder T_{p_j}U$ for $U\subset (Q/G)$ is defined by
\begin{equation}\label{Aoperator}
\bra T^{*}\hat L_{(g\,A_i)}(j)D_3\mathcal{L}_d(a),\delta  p_j\ket:=\bra  T_{gA}^*L_{g^{-1}}D_3\mathcal{L}_d(a),\bra D_iA,\delta  p_j\ket\ket,
\end{equation}
where $a\in U\times U\times G$, $a:=(p_0, p_1,g\,A(p_0,p_1))$, $i=\lc1,2\rc$ and $j=i-1$  for each $i.$  Let us define $\mu_n:=T_{W_n}^*R_{W_n^{-1}}(T_{W_nA_{n}}^{*}R_{A_{n}^{-1}}D_3\mathcal{L}_d(a_{n}))\in \dal$. It is easy to see that \eqref{DEoMb} can be rewritten in its dual version as
\begin{equation}\label{EuPoinGeneral}
\mu_{n}=\mbox{Ad}_{_{W_{n-1}}}^*\,\mu_{n-1}.
\end{equation}
Next, we introduce constraints in our picture by considering the augmented Lagrangian   \eqref{AugLag} instead of $\mathcal{L}_d$, which inserted into \eqref{Vari} leads to \begin{align}
0=&D_1\mathcal{L}_d(a_n)+D_2\mathcal{L}_d(a_{n-1})+T^{*}\hat L_{(WA_1)}(n)D_3\mathcal{L}_d(a_n)\nonumber\\
&+T^{*}\hat L_{(WA_2)}(n-1)D_3\mathcal{L}_d(a_{n-1})+\lambda_{\alpha}^n\lc D_1\chi_d^{\alpha}(a_n)+T^{*}\hat L_{(WA_1)}(n)D_3\chi_d^{\alpha}(a_n)\rc\nonumber\\
&+\lambda_{\alpha}^{n-1}\lc D_2\chi_d^{\alpha}(a_{n-1})+T^{*}\hat L_{(WA_2)}(n-1)D_3\chi_d^{\alpha}(a_{n-1})\rc,\nonumber\\ \label{DiscEqofMotionSimples}\\
 \mu_n=&\mbox{Ad}^*_{W_{n-1}}\mu_{n-1}-\lambda_{\alpha}^n\varepsilon_n^{\alpha}+\lambda_{\alpha}^{n-1}\mbox{Ad}^*_{W_{n-1}}\varepsilon_{n-1}^{\alpha},\nonumber\\
0=&\chi_d^{\alpha}(a_n),\nonumber
\end{align} for $n=1,...,N-1$, where we denote $\varepsilon_n^{\alpha}:=T_{W_n}^*R_{W_n^{-1}}(T_{W_nA_{n}}^{*}R_{A_{n}^{-1}}D_3\chi_d^{\alpha}(a_{n}))\in\dal$.  

\end{proof}

To obtain the discrete-time equations \eqref{DiscEqofMotionSimples} we used the approach studied in \cite{MMM}. That is,  by using a discrete connection instead of deriving the local description of the curvature terms as in \cite{FZ2}. This approach automatically gives preservation of momentum and symplecticity since we employ a variational approach (see \cite{MMM} for further details).

Note that the first equation in \eqref{DiscEqofMotionSimples} represents a discrete-time version of the second equation in \eqref{LagPoincareGeneral} (or equivalently \eqref{LagPoinLocal} in a local description) where the curvature terms are included in the terms that come from \eqref{Aoperator}. The second equation represents the (constrained) Euler-Poincar\'e part (first equation in \eqref{DiscEqofMotionSimples}) in \eqref{LagPoincareGeneral}, or \eqref{LagPoinLocal} in the local representation.

Next, define $M_{(1,3)}(a_{n}):=D_1\chi_d^{\alpha}(a_n)+T^{*}\hat L_{(WA_1)}(n)D_{3}\chi_d^{\alpha}(a_n)$, where $D_i$ denotes the partial derivative with respect to the $i$-th component, while $D_{ij}=D_iD_j=D_jD_i$. Equations \eqref{DiscEqofMotionSimples} determine a numerical integrator giving rise to a unique (local) variational flow given an initial value on $U\times U\times G\times\mathbb{R}^{m}$ under the following algebraic conditions:

\begin{proposition}\label{proposition1} 
Let $\mathcal{M}_{d}$ be a regular submaniolfd of $(U\times U\times G)$ given by $$\mathcal{M}_{d}=\{a_n\in (U\times U\times G) \big|\; \chi_d^{\alpha}(a_n)=0\},$$
where $a_n$ is defined in \eqref{a_n}. If the matrix \begin{equation*}\label{regu}
\begin{bmatrix}
 D_{12}\widetilde{\mathcal{L}}_{d}(a_n,\lambda^{n}) &  D_{3}\left(T^{*}\widehat{L}_{(WA_1)}(n)D_3\widetilde{\mathcal{L}}_{d}(a_n,\lambda^n)\right) & M_{(1,3)}(a_{n}) \\
D_{2}\mu_n(a_n) &  D_3\mu_n(a_n) &\epsilon_{n}^{\alpha}(a_n) \\
(D_{2}\chi_{d}^{\alpha}(a_n))^{T} & (D_{3}\chi_{d}^{\alpha}(a_n))^{T} & 0 \\
\end{bmatrix}
\end{equation*} 
is non singular for all $a_n\in {\mathcal M}_d$, there exists a neighborhood $\mathcal{U}_k\subset \mathcal M_d\times \R^{m}$ of $(a_n^{*}, \lambda^0_{\alpha*})$ satisfying equations \eqref{DiscEqofMotionSimples}, and an unique (local) application $\Upsilon_{\mathcal{L}_d}:\mathcal{U}_k\subset\mathcal M_d\times \R^{m}\rightarrow\mathcal M_d\times \R^{m}$ such that
\begin{align*}
\Upsilon_{\mathcal{L}_d}(a_{n}, \lambda^0_{\alpha})=&(a_{n+1}, \lambda^1_{\alpha}).
\end{align*}
\end{proposition}

\textit{Proof:} It is a direct consequence of the implicit function theorem applied to equations \eqref{DiscEqofMotionSimples}. \hfill$\square$

\begin{remark}
Note that the regularity condition given in Proposition \ref{proposition1} represents a first order discretization of the regularity condition for first order vakonomic systems presented, for instance, in \cite{Bloch} (Section $7.3$, Equation $7.3.5$), \cite{Benito2}, \cite{LeoJGM}, \cite{Jorge}, \cite{vako2} and \cite{vakoalg}. In general such a condition for continuous time systems is expressed as a $2\times 2$ block-matrix  $\begin{bmatrix}
A &  B \\
C&  D \\
\end{bmatrix}$ where $A$ corresponds to the matrix giving the classical hyper-regularity condition for the equivalence between Lagrangian and Hamiltonian formalism in mechanics by means of the Legendre transform (this corresponds with the first two entries in the first row of the matrix in Proposition \ref{proposition1}). The sub-matrix $C$ corresponds to the partial derivative of the constraints with respect to the velocities, as does the sub-matrix $B$, and $D$ is the null sub-matrix. Taking into account the split into vertical and horizontal variables, it is easy to see the similarities of the regularity condition of the continuous-time and discrete-time systems.\hfill$\diamond$
\end{remark}

\begin{remark}
In the case $Q=G$, \eqref{EuPoinGeneral} reduces to the usual discrete Euler-Poincar\'e equations \cite{MPS}. In this case $A=e$, and therefore $a_n=W_n$. Thus, \eqref{DEoMb} becomes
\[
T^{*}	_{W_{n-1}}L_{W_{n-1}^{-1}}\mathcal{L}_d^{\prime}(W_{n-1})- T^*_{W_{n}}R_{W_{n}^{-1}}\mathcal{L}_d^{\prime}(W_{n})=0,
\]
where $^{\prime}$ denotes the derivative with respect to $W$. Setting $\mu_n:=T^*_{W_{n}}R_{W_{n}^{-1}}\mathcal{L}_d^{\prime}(W_{n})$, $\mu_n\in\dal$, we arrive at the discrete Lie-Poisson equations 
{\rm$\mu_{n}=\mbox{Ad}_{_{W_{n-1}}}^*\,\mu_{n-1}$}.\hfill$\diamond$
\end{remark}

\subsubsection{Variational integrators for constrained HO Lagrange-Poincar\'e equations:} Next, we consider the HO case \eqref{HOcase}. According to \eqref{HOSplit} and \eqref{Bijecton}, we can find local coordinates $[q_0,q_1,...,q_k]_G$ in $U^{(k+1)}\times\,G^{k}$ given by
\begin{equation}\label{akhigherorder}
\lp p_0,...,p_k,\widetilde g_0,\widetilde g_1,....,\widetilde g_{k-1}\rp,
\end{equation}
where by $\widetilde g_i:= g_i^{-1}g_{i+1}A(p_i,p_{i+1})$ we denote the element of the $i$-th copy of $\widetilde G$, for $i=0,...,k-1$ and $U^{(k+1)}$ denotes $(k+1)$-copies of the neighborhood $U\subset Q/G$. The variation of the $i$-th copy of $Q/G$ is given as before by $\delta p_i$, for $i=0,...,k$; while the variation of $\widetilde g_i$ is given by 
\begin{equation}\label{VarTildeg}
\begin{split}
\delta\widetilde g_i=&-\eta_iW_iA_i+W_i\eta_{i+1}\,A_i+W_i\bra D_1A_i,\delta p_i\ket+W_i\bra D_2A_i,\delta p_{i+1}\ket,
\end{split}
\end{equation}
where we have set  $W_i=g_i^{-1}g_{i+1}\in G$, $\eta_i=g_i^{-1}\delta g_i\in\al$ and $A_i:=A(p_i,p_{i+1}).$

In the HO case, given the grid $\{t_{n}=nh\mid n=0,\ldots,N\}$,
with $Nh=T$, the discrete path space is determined by 
\[
\mathcal{C}_d\lp U^{(k+1)}\times \widetilde{G}^{k}\rp:=\lc \tilde\gamma_d:\lc t_n\rc_{n=0}^N\Flder U^{(k+1)}\times \widetilde{G}^{k}\rc. 
\]
The discrete space will be identified with its image, i.e. $\tilde\gamma_d(t_n)=\lc \tilde a_n\rc_{n=0}^N$, where we employ the notation 
\begin{equation}\label{antilde}
\tilde a_n:=(p_n,p_{n+1},...,p_{n+k},\tilde g_n,\tilde g_{n+1},...,\tilde g_{n+k-1}).
\end{equation} 

We see that $\tilde a_n$ is a $(2k+1)$-tuple with $2k+1$ elements. 
This discrete path space is isomorphic to the
smooth product manifold which consists of $N+1$ copies of $U^{(k+1)}\times \widetilde{G}^{k}$ (which locally is isomorphic to $N+1$ copies of $(Q/G)^{(k+1)}\times_{Q/G}\times \tilde G^{k}$).

Let us define the discrete action sum associated with the HO  Lagrangian $\mathcal{L}_{d}$ as $\mathcal{S}_d:\mathcal{C}_d\lp U^{(k+1)}\times \widetilde{G}^{k}\rp\Flder\R$ given by
\begin{equation}\label{DiscActionHO}
\mathcal{S}_d(\tilde\gamma_d)=\sum_{n=0}^{N-k}\mathcal{L}_d([q_n,q_{n+1},...,q_{n+k}]_G)=\sum_{n=0}^{N-k}\mathcal{L}_d(\tilde a_n)
\end{equation}
where the second equality is established at a local level.

The \textit{discrete constrained HO variational problem} associated with \eqref{HOcase}, consists of finding a discrete path $\tilde\gamma_d\in\mathcal{C}_{d}(U^{(k+1)}\times \widetilde{G}^{k})$, given fixed boundary conditions, which extremizes the discrete action sum \eqref{DiscActionHO} subject to the discrete constraints $\chi^{\alpha}_{d}$. This constrained optimization problem is equivalent to studying the (unconstrained) optimization problem for the \textit{augmented Lagrangian} $\widetilde{\mathcal{L}}_{d}: H^{(k+1,k)}\times\R^{m}\Flder\R$ given by 
\begin{equation}\label{AugLagHO}
\widetilde{\mathcal{L}}_{d}([q_n,...,q_{n+k}]_G,\lambda_n):=\mathcal{L}_{d}([q_n,...,q_{n+k}]_G)+\lambda_{\alpha}^n\chi_d^{\alpha}([q_n,...,q_{n+k}]_G),
\end{equation}
 where 
$\lambda^n_{\alpha}=(\lambda_1^n,...,\lambda_m^n)\in\R^m$ are Lagrange multipliers, and its associated action sum is given by \begin{equation}\label{AugmentedActionSumHO}
\mathcal{S}_d(\hat\gamma_{d})=\sum_{n=0}^{N-k}\widetilde{\mathcal{L}}_d([q_n,...,q_{n+k}]_G,\lambda_n)=\sum_{n=0}^{N-k}\widetilde{\mathcal{L}}_d(\tilde a_n,\lambda_{\alpha}^n),
\end{equation}
where again the second equality is given at a local level and $\hat\gamma_d\in\mathcal{C}_{d}( U^{(k+1)}\times \widetilde{G}^{k}\times\R^m):=\{\hat\gamma_{d}:\{t_{n}\}_{n=0}^{N}\ra U^{ (k+1)}\times \widetilde{G}^{k}\times \R^m\}$ is the discrete augmented path space.

Regarding the endpoint conditions, we shall consider $q_{(0,k-1)}=(p_{(0,k-1)},g_{(0,k-1)})$ and $q_{(N-k+1,N)}=(p_{(N-k+1,N)},g_{(N-k+1,N)})$ fixed, where $q_{(0,k-1)}=\lc q_0,q_1,...,q_{k-1}\rc$, $q_{(N-k+1,N)}=\lc q_{N-k+1},q_{N-k+2},...,q_N\rc$, and analogously for any sequence. In terms of variations this implies $\delta p_{(0,k-1)}=\delta p_{(N-k+1,N)}=0$ and $\delta g_{(0,k-1)}=\delta g_{(N-k+1,N)}=0$, the latter leading to $\eta_{(0,k-1)}=\eta_{(N-k+1,N)}=0$. Furthermore, the Lagrange multipliers are set freely as in the first order case.

We establish  the result in the following theorem, where the discrete constrained HO Lagrange-Poincar\'e equations are obtained. As in the case of Theorem \ref{TheoremDisc1}, our proof strategy consists in studying the unconstrained problem \eqref{DiscActionHO}, and afterwards adding the constraints \eqref{AugLagHO}.

\begin{theorem}\label{HOTheorem}
A discrete sequence $\lc \tilde a_n,\lambda^n_{\alpha}\rc_{n=0}^N\in \mathcal{C}_{d}(U^{(k+1)}\times \widetilde{G}^{k}\times\R^m)$ is an extremum of the action sum \eqref{AugmentedActionSumHO}, with respect to  variations $\delta[q_n,...,q_{n+k}]_G$ defined in \eqref{VarTildeg} and endpoint conditions expressed above, if it is a solution of the discrete constrained HO Lagrange-Poincar\'e equations  \eqref{EqHigherOrder}.
\end{theorem}

\begin{proof}
In the proof we will employ the index $i$ for the $k+1$ first elements, i.e. the $p$ coordinates, and the index $z$ for the last $k$, i.e. the $\widetilde g$ coordinates. Taking variations in \eqref{DiscActionHO}, according to the endpoint conditions detailed above and the variations \eqref{VarTildeg} we obtain:
\begin{align}
\delta \sum_{n=0}^{N-k}\mathcal{L}_d(\tilde a_n)=&\sum_{n=0}^{N-k}\lp\sum_{i=1}^{k+1}\bra D_i\mathcal{L}_d(\tilde a_n),\delta p_i\ket+\sum_{\ti=k+2}^{2k+1}\bra D_{\ti}\mathcal{L}_d(\tilde a_n),\delta\widetilde g_{\ti}\ket \rp\label{VarHigherOrder}\\
=&\sum_{n=0}^{N-k}\sum_{i=1}^{k+1}\bra D_i\mathcal{L}_d(\tilde a_n),\delta p_i\ket\nonumber\\ 
&+\sum_{n=0}^{N-k}\sum_{\ti=k+2}^{2k+1}\left(\bra T^*_{W_{\ti}}R_{W_{\ti}^{-1}}(T^*_{W_{\ti}A_{\ti}}R_{A_{\ti}^{-1}}D_{\ti}\mathcal{L}_d(\tilde a_n)),-\eta_{\ti}\ket\right.\nonumber\\
&\left.\qquad\qquad\qquad\qquad+ \bra T^*_{W_{\ti}}L_{W_{\ti}^{-1}}(T^*_{W_{\ti}A_{\ti}}R_{A_{\ti}^{-1}}D_{\ti}\mathcal{L}_d(\tilde a_n)),\eta_{\ti+1}\ket\right.\nonumber\\
&\left.\qquad\qquad\qquad\qquad\qquad+\bra T^*_{W_{\ti}A_{\ti}}L_{W_{\ti}^{-1}} D_{\ti}\mathcal{L}_d(\tilde a_n),\bra D_1A_{\ti},\delta p_{\ti}\ket\ket\right.\nonumber\\
&\left.\qquad\qquad\qquad\qquad\qquad+ \bra T^*_{W_{\ti}A_{\ti}}L_{W_{\ti}^{-1}}D_{\ti}\mathcal{L}_d(\tilde a_n),\bra D_2A_{\ti},\delta p_{\ti+1}\ket\ket\right) \nonumber
\end{align} where we  have employed \eqref{VarTildeg}.  Next, we assume that the $z$-th component, for $z=k+2,...,2k+1$,  is labeled by $n+z-k-2$ and rearranging the sum above after taking into account the endpoint conditions we obtain:

\[
\begin{split}
\delta \sum_{n=0}^{N-k}\mathcal{L}_d(\tilde a_n)=&\sum_{n=k}^{N-k}\bra\sum_{i=1}^{k+1}D_i\mathcal{L}_d(\tilde a_{n-i+1}),\delta p_n\ket\\
+&\sum_{n=k}^{N-k}\bra-\sum_{\ti=k+2}^{2k+1} T^*_{W_{n}}R_{W_{n}^{-1}}(T^*_{W_{n}A_{n}}R_{A_{n}^{-1}}D_{\ti}\mathcal{L}_d(\tilde a_{n-\ti+k+2})),\eta_n\ket\\
+&\sum_{n=k}^{N-k}\bra\sum_{\ti=k+2}^{2k+1} T^*_{W_{n-1}}L_{W_{n-1}^{-1}}(T^*_{W_{n-1}A_{n-1}}R_{A_{n-1}^{-1}}D_{\ti}\mathcal{L}_d(\tilde a_{n-\ti+k+1})),\eta_{n}\ket\\
+&\sum_{n=k}^{N-k}\bra\sum_{\ti=k+2}^{2k+1}\lp T^*\hat L_{_{WA_1}}(n)D_{\ti}\mathcal{L}_d(\tilde a_{n-\ti+k+2})\right.\\&\hspace{4cm}\left.+ T^*\hat L_{_{WA_2}}(n-1)D_{\ti}\mathcal{L}_d(\tilde a_{n-\ti+k+1})\rp,\delta p_{n}\ket,
\end{split}
\]
where the operator $T^*\hat L_{_{WA_i}}(n)$ is defined in \eqref{Aoperator}. Equating this variation to zero and considering that $\delta p_n$ and $\eta_n$ are free for $k\leq n\leq N-k$, we arrive at the discrete equations of motion:
\begin{align*}
0=&\sum_{i=1}^{k+1}D_i\mathcal{L}_d(\tilde a_{n-i+1})\\
&+\sum_{\ti=k+2}^{2k+1}\lp T^*\hat L_{_{WA_1}}(n)D_{\ti}\mathcal{L}_d(\tilde a_{n-\ti+k+2})+ T^*\hat L_{_{WA_2}}(n-1)D_{\ti}\mathcal{L}_d(\tilde a_{n-\ti+k+1})\rp,\\
0=&\sum_{\ti=k+2}^{2k+1}\lp T^*_{W_{n}}R_{W_{n}^{-1}}(T^*_{W_{n}A_{n}}R_{A_{n}^{-1}}D_{\ti}\mathcal{L}_d(\tilde a_{n-\ti+k+2}))\right.\\
&\left.- T^*_{W_{n-1}}L_{W_{n-1}^{-1}}(T^*_{W_{n-1}A_{n-1}}R_{A_{n-1}^{-1}}D_{\ti}\mathcal{L}_d(\tilde a_{n-\ti+k+1}))\rp,
\end{align*}
for $k\leq n\leq N-k$. The second equation may be rewritten in a more compact way in its dual version by making the following identifications

\begin{itemize}
\item[-] $\tilde\mu^{\ti}_n:=D_{\ti}\mathcal{L}_d(\tilde a_{n-\ti+k+2})\in T^*_{W_nA_n}G$ for $k+2\leq \ti\leq 2k+1$, 
\item[-] $\displaystyle{\tilde M_n:=\sum_{\ti=k+2}^{2k+1}\tilde\mu^{\ti}_n\in T^*_{W_nA_n}G}$,
\item[-] $M_n:=T^*_{W_{n}}R_{W_{n}^{-1}}(T^*_{W_{n}A_{n}}R_{A_{n}^{-1}}\tilde M_n)\in\dal$,
\end{itemize}
which leads to the equation $M_n=\mbox{Ad}_{W_{n-1}}^*M_{n-1},\quad k\leq n\leq N-k$.

Next, introducing constraints into our picture by considering the augmented Lagrangian  \eqref{AugLagHO} we find the discrete constrained HO Lagrange-Poincar\'e equations  

\begin{align}
0=&\sum_{i=1}^{k+1}D_i\mathcal{L}_d(\tilde a_{n-i+1})+T^*\hat L_{_{WA_1}}(n)\sum_{\ti=k+2}^{2k+1}D_{\ti}\mathcal{L}_d(\tilde a_{n-\ti+k+2})\nonumber\\
&+ T^*\hat L_{_{WA_2}}(n-1)\sum_{\ti=k+2}^{2k+1}D_{\ti}\mathcal{L}_d(\tilde a_{n-\ti+k+1})
+\sum_{i=1}^{k+1}\lambda_{\alpha}^{n-i+1}D_i\chi_d^{\alpha}(\tilde a_{n-i+1})\nonumber\\
&+T^*\hat L_{_{WA_1}}(n)\sum_{\ti=k+2}^{2k+1}\lambda_{\alpha}^{n-\ti+k+2}D_{\ti}\chi_d^{\alpha}(\tilde a_{n-\ti+k+2})\nonumber\\
&+T^*\hat L_{_{WA_2}}(n-1)\sum_{\ti=k+2}^{2k+1}\lambda_{\alpha}^{n-\ti+k+1}D_{\ti}\chi_d^{\alpha}(\tilde a_{n-\ti+k+1}),\nonumber\\\label{EqHigherOrder}\\
M_n=&\mbox{Ad}^*_{W_{n-1}}M_{n-1}-\sum_{\ti=k+2}^{2k+1}\lambda_{\alpha}^{n-\ti+k+2}\,\varepsilon^{\alpha}_{(n,\ti)}+\mbox{Ad}^*_{W_{n-1}}\sum_{\ti=k+2}^{2k+1}\lambda_{\alpha}^{n-\ti+k+1}\,\varepsilon^{\alpha}_{(n-1,\ti)},\nonumber\\\nonumber\\
0=&\chi^{\alpha}_{d}(\tilde a_n),\nonumber
\end{align}
for $k\leq n\leq N-k$, where we define
\[
\varepsilon^{\alpha}_{(n,\ti)}:=T^*_{W_{n}}R_{W_{n}^{-1}}(T^*_{W_{n}A_{n}}R_{A_{n}^{-1}}D_{\ti}\chi_d^{\alpha}(\tilde a_{n-\ti+k+2}))\in\al^{*}.
\]

\end{proof}
As in the first order case, a direct consequence of the implicit function theorem applied to \eqref{EqHigherOrder} is the existence of the (local) variational flow for the numerical method.

Denoting $M_{(1,k+2)}(\tilde{a}_{n}):=D_1\chi_d^{\alpha}(\tilde a_n)+T^{*}\hat L_{(WA_1)}(n)D_{k+2}\chi_d^{\alpha}(\tilde a_n)$ we arrive at the following proposition.
\begin{proposition}\label{proposition2} 
Let $\widetilde{\mathcal{M}}_{d}$ be a regular submanifold of $U^{(k+1)}\times G^{k}$ given by $$\widetilde{\mathcal{M}}_{d}=\{\tilde{a}_n\in U^{(k+1)}\times G^{k} \big|\; \chi_d^{\alpha}(\tilde{a}_{n-j})=0 \hbox{ for all }j=0,\ldots,k\}$$ where $\tilde{a}_{n}$ is of the form \eqref{antilde}.

If the matrix  \begin{equation*}
\begin{bmatrix}
 D_{(1,k+1)}\widetilde{\mathcal{L}}_{d}(\tilde a_n,\lambda^{n}) &  D_{2k+1}\left(T^{*}\widehat{L}_{(WA_1)}(n)D_{k+2}\widetilde{\mathcal{L}}_{d}(\tilde a_n,\lambda^n)\right) & M_{(1,k+2)}(\tilde{a}_{n})\\
D_{k+1}M_n(\tilde a_n) &  D_{2k+1}M_n(\tilde a_n) &\varepsilon_{(n,k+2)}^{\alpha}(\tilde a_n)\\
(D_{k+1}\chi_{d}^{\alpha}(\tilde a_n))^{T} & (D_{2k+1}\chi_{d}^{\alpha}(\tilde a_n))^{T} & 0 \\
\end{bmatrix}
\end{equation*} 
is non singular for all $\tilde{a}_n\in {\widetilde{\mathcal M}}_d$, there exists a neighborhood $\mathcal{V}_k\subset \widetilde{\mathcal M}_d\times k\R^{m}$ of $\gamma^{*}=(\tilde{a}_{n-k}^{*},\ldots,\tilde{a}_{n-1}^{*}, \lambda^{n-k}_{\alpha*},\ldots,\lambda^{n-1}_{\alpha*})$ satisfying equations \eqref{EqHigherOrder}, and an unique (local) application $\widetilde{\Upsilon}_{\mathcal{L}_d}:\mathcal{V}_k\subset\widetilde{\mathcal M}_d\times k\R^{m}\rightarrow\widetilde{\mathcal M}_d\times k\R^{m}$ such that
\begin{align*}
\widetilde{\Upsilon}_{\mathcal{L}_d}(\tilde{a}_{n-k},\ldots,\tilde{a}_{n-1}, \lambda^{n-k}_{\alpha},\ldots,\lambda^{n-1}_{\alpha})=&(\tilde{a}_{n-k+1},\ldots,\tilde{a}_{n}, \lambda^{n-k+1}_{\alpha},\ldots,\lambda^{n}_{\alpha}).
\end{align*}
\end{proposition}

Observe that when $k=1$ equations  \eqref{EqHigherOrder} are the discrete constrained Lagrange-Poincar\'e equations \eqref{DiscEqofMotionSimples} and the regularity condition given in Proposition \ref{proposition2} is the one obtained in Proposition \ref{proposition1}.

\begin{remark}
In \cite{CoMdDZu2013} it has been shown that under a
regularity condition equivalent to the one given in Proposition \ref{proposition2}, the discrete constrained system 
preserves the symplectic $2$-form (see Remark 3.4 in \cite{CoMdDZu2013}). Therefore the methods that we are deriving in this work are automatically
symplectic methods. Moreover, under a group of symmetries preserving
the discrete Lagrangian and the constraints, we additionally obtain
momentum preservation. In the case when the principal bundle is a trivial bundle, and therefore the terms associated with the connection and curvature are zero, we obtain the same results as \cite{CJdD}. \hfill$\diamond$ \end{remark}

\section{Application to optimal control of underactuated systems}\label{OCP}

Underactuated mechanical system are controlled mechanical systems  where the number of the control inputs is strictly less
than the dimension of the configuration space. In this section we consider dynamical optimal control problems for a class of underactuated mechanical systems determined by Lagrangian systems on principal bundles.  

We assume that we are only allowed to have control systems that are controllable, that is, for any two
points $q_0$ and $q_T$ in the configuration space, there exists an
admissible control defined on some interval $[0,T]$ such that
the system with initial condition $q_0$ reaches the point $q_T$ in
time $T$ (see \cite{Bloch} for more details).

Let $L:TQ\to\mathbb{R}$ be a $G$-invariant Lagrangian inducing a reduced Lagrangian $\mathcal{L}:M\to\mathbb{R}$ where $M:=T(Q/G)\times_{Q/G}\widetilde{\mathfrak{g}}$ and $(p,\dot p,\sigma)$ are local coordinates on an open set $\Omega\subset M$. Consider the control manifold $\frak{U}\subseteq\R^{r}$ where $r<\dim Q$ and $u\in \frak{U}$ is the control input 
(control parameter) which in coordinates reads $u=(u_1,\ldots,u_r)\in\mathbb{R}^{r}$. 

We denote by $\Gamma(M^*)$ the space of sections of a smooth manifold $$M^{*}:=T^{*}(Q/G)\times_{Q/G}\widetilde{\mathfrak{g}}^{*}$$  and consider a set of linearly independent sections $B^{a}=\{(\eta^{a},\widetilde{\eta}^{a})\}\in\Gamma(M^{*})$, such that  $\eta^{a}([q]_{G})\in T^{*}_{[q]_{G}}(Q/G)$; $\widetilde{\eta}^{a}([q]_{G})\in\widetilde{\mathfrak{g}}^{*}$ for $a=1,\ldots,r$ and $[q]_G\in \tau(\Omega)\subset Q/G$, where $\tau:M\Flder Q/G$. Therefore $\eta^{a}\oplus\widetilde{\eta}^{a}\in\Gamma(M^{*})$.

\begin{definition}

The \textit{reduced controlled Euler-Lagrange equations} or \textit{controlled Lagrange-Poincar\'e equations} are   \begin{subequations}\label{control}
\begin{align}
\frac{D}{Dt}\left(\frac{\partial \mathcal{L}}{\partial \dot p}\right)-\frac{\partial \mathcal{L}}{\partial p}+\Big{\langle}\frac{\partial \mathcal{L}}{\partial\sigma},i_{\dot{p}}\widetilde{\mathcal{B}}\Big{\rangle}&=u_a\eta^{a}([q]_{G}),\label{controla}\\
\left(\frac{D}{Dt}-\mbox{ad}_{\sigma}^{*}\right)\frac{\partial \mathcal{L}}{\partial \sigma}&=u_a\widetilde{\eta}^a([q]_{G}).\label{controlb}
\end{align}
\end{subequations} A  \textit{controlled Lagrange-Poincar\'e system} is a controlled mechanical systems whose dynamics is given by the controlled Lagrange-Poincar\'e equations \eqref{control}. 
\end{definition}

We refer to a \textit{controlled decoupled  Lagrange-Poincar\'e system} when equations \eqref{controla}-\eqref{controlb} can be written as a system of equations of the form  
\begin{subequations}\label{control1}
\begin{align}
\Big{\langle}\frac{D}{Dt}\left(\frac{\partial \mathcal{L}}{\partial\dot{p}}\right)-\frac{\partial \mathcal{L}}{\partial p}-\Big{\langle}\frac{\partial \mathcal{L}}{\partial\sigma}; i_{\dot{p}}\widetilde{\mathcal{B}}\Big{\rangle},\eta_{a}\Big{\rangle}+\Big{\langle}\left(\frac{D}{Dt}-\mbox{ad}_{\sigma}^{*}\right)\frac{\partial \mathcal{L}}{\partial\sigma},\widetilde{\eta}_{a}\Big{\rangle}=&\,u_{a},\label{control1a}\\
\Big{\langle}\frac{D}{Dt}\left(\frac{\partial \mathcal{L}}{\partial\dot{p}}\right)-\frac{\partial \mathcal{L}}{\partial	 p}-\Big{\langle}\frac{\partial \mathcal{L}}{\partial\sigma}; i_{\dot{p}}\widetilde{\mathcal{B}}\Big{\rangle},\eta_{\alpha}\Big{\rangle}+\Big{\langle}\left(\frac{D}{Dt}-\mbox{ad}_{\sigma}^{*}\right)\frac{\partial \mathcal{L}}{\partial\sigma},\widetilde{\eta}_{\alpha}\Big{\rangle}=&\,0, \label{control1b}
\end{align}
\end{subequations} that is, a controlled Lagrange-Poincar\'e system is written as a control system showing which configurations are actuated and which ones unactuated.

The next Lemma shows that a controlled Lagrange-Poincar\'e system always permits a description for the controlled dynamics as a controlled decoupled  Lagrange-Poincar\'e system.

\begin{lemma}\label{lemmadec}
A controlled Lagrange-Poincar\'e system defined by \eqref{control} is equivalent to the controlled decoupled Lagrange-Poincar\'e system described by \eqref{control1}.

\end{lemma}
\begin{proof}

Given that $B^{a}=\{(\eta^{a},\widetilde{\eta}^{a})\},$
are independent elements of
$\Gamma(M^*)$ we complete $B^{a}$ to be a basis of
$\Gamma(M^{*}),$ i.e.
$\{B^{a},B^{\alpha}\}$, and take its dual basis
$\{B_{a},B_{\alpha}\}$ on
$\Gamma(M).$ If we set 
$B_{a}=\{(\eta_{a},\widetilde{\eta}_{a})\}$ and
$B_{\alpha}=\{(\eta_{a},\widetilde{\eta}_{\alpha})\},$
where $\eta_{a},\eta_{\alpha}\in\mathfrak{X}(Q/G)$ and
$\widetilde{\eta}_{a},\widetilde{\eta}_{\alpha}\in\Gamma(\widetilde{\mathfrak{g}}),$ we  obtain the relationships
\[ 
\begin{split}
\bra\eta^a,\eta_b\ket&=\delta^a_b,\,\bra\eta^a,\tilde\eta_b\ket=\bra\eta^a,\eta_{\beta}\ket=\bra\eta^a,\tilde\eta_{\beta}\ket=0,\\
\bra\tilde\eta^a,\tilde\eta_b\ket&=\delta^a_b,\,\bra\tilde\eta^a,\eta_{\beta}\ket=\bra\tilde\eta^a,\tilde\eta_{\beta}\ket=0,\\
\bra\eta^{\alpha},\eta_{\beta}\ket&=\delta^{\alpha}_{\beta},\,\bra\tilde\eta^{\alpha},\tilde\eta_{\beta}\ket=0,\\
\bra\tilde\eta^{\alpha},\tilde\eta_{\beta}\ket&=\delta^{\alpha}_{\beta}.
\end{split}
\]
Coupling \eqref{controla} to $\eta_a$ and  \eqref{controlb} to $\tilde\eta_a$, and adding up the results we obtain \eqref{control1a}. Equivalently, if we couple  \eqref{controla} to $\eta_{\alpha}$ and  \eqref{controlb} to $\tilde\eta_{\alpha}$, and add up the resultants we obtain \eqref{control1b}.
\end{proof}

\begin{remark}
Observe that \eqref{control1a} provides an expression of the control inputs as a function on the second-order tangent bundle $M^{(2)}$ locally described by coordinates  $(p,\dot p,\ddot p,\sigma, \dot \sigma)$, 
\begin{small}
\begin{equation}\label{ControlesLocales}
u_a=F_a(p,\dot p,\ddot p,\sigma, \dot \sigma)
=\Big{\langle}\frac{D}{Dt}\left(\frac{\partial \mathcal{L}}{\partial\dot{p}}\right)-\frac{\partial \mathcal{L}}{\partial p}-\Big{\langle}\frac{\partial \mathcal{L}}{\partial\sigma}; i_{\dot{p}}\widetilde{\mathcal{B}}\Big{\rangle},\eta_{a}\Big{\rangle}+\Big{\langle}\left(\frac{D}{Dt}-\mbox{ad}_{\sigma}^{*}\right)\frac{\partial \mathcal{L}}{\partial\sigma},\widetilde{\eta}_{a}\Big{\rangle}.
\end{equation}
\end{small}
\hfill$\diamond$
\end{remark}

Next we consider an optimal control problem.

\begin{definition}[Optimal control problem]
 Find a trajectory $\gamma(t)=(p(t), \sigma(t),$ $u(t))$ of the
state variables and control inputs satisfying \eqref{control},
subject to boundary conditions
 $(p(0),\dot p(0), \sigma(0))$ and $(p(T),\dot p(T), \sigma(T))$, and minimizing the cost
functional
\[
{\mathcal J}(s^{(2,1)}, u)=\int_{0}^{T} C(s^{(2,1)}(t),  u(t))\, dt
\]
 for a cost function $C:M\times \frak{U}\to\mathbb{R}$.
 \end{definition}
 
   Solving the optimal control problem is equivalent to solving a constrained second-order
 variational problem \cite{BlochCrouch2}, with  Lagrangian $\hat{\mathcal{L}}:M^{(2)}\Flder\R$  locally described by
\begin{equation}\label{HatLag}
\widehat{\mathcal{L}}(s^{(2,1)}):=C\left(s^{(1,0)} ,F_{a}(s^{(2,1)})\right),
\end{equation}
where $C$ is the cost function and $F_a$ is defined in \eqref{ControlesLocales}; and subject to the constraints
 $\chi^{\alpha}:M^{(2)}\Flder\R$ given by   
\begin{equation}\label{Cons}
\mathcal{\chi}^{\alpha}(s^{(2,1)})=\Big{\langle}\frac{d}{dt}\left(\frac{\partial \mathcal{L}}{\partial\dot{p}}\right)-\frac{\partial \mathcal{L}}{\partial	 p}-\Big{\langle}\frac{\partial \mathcal{L}}{\partial\sigma}; i_{\dot{p}}\widetilde{\mathcal{B}}\Big{\rangle},\eta_{\alpha}\Big{\rangle}+\Big{\langle}\left(\frac{D}{Dt}-\mbox{ad}_{\sigma}^{*}\right)\frac{\partial \mathcal{L}}{\partial\sigma},\widetilde{\eta}_{\alpha}\Big{\rangle},
\end{equation} equivalent to equation \eqref{control1b}.

 Then, given boundary conditions, \textcolor{blue}{necessary optimality conditions} for the optimal control problem are determined by the solutions of the constrained second-order Lagrange-Poincar\'e equations for the Lagrangian  \eqref{HatLag} subject to \eqref{Cons}. The resulting equations of motion are a set of combined third order and fourth order ordinary differential equations. 

Motivated by the examples that we study in the next section, we restrict ourself to a particular class of these control problems where we assume \textit{full controls in the base manifold $Q/G$}, that is, using Lemma \ref{lemmadec}, we consider the controlled Lagrange-Poincar\'e equations, in a local trivialization $\pi_{U}:U\times G\to U$ of the principal bundle $\pi:Q\to Q/G$, i.e. \begin{subequations}\label{control2}\begin{align}
\frac{d}{dt}\frac{\partial\mathcal{L}}{\partial\sigma^{\beta}}-\frac{\partial\mathcal{L}}{\partial\sigma^{\beta}}(C_{\gamma\beta}^{\delta}\sigma^{\gamma}-C_{\gamma\beta}^{\delta}A_{\epsilon}^{\gamma}\dot{p}^{\epsilon})&=0,\label{control2a}\\
\frac{\partial\mathcal{L}}{\partial p^{a}}-\frac{d}{dt}\frac{\partial\mathcal{L}}{\partial \dot{p}^{a}}-\frac{\partial\mathcal{L}}{\partial\sigma^{b}}(B_{ca}^{b}\dot{p}^{c}+C_{de}^{b}\sigma^{d}A_{a}^{e})&=u_{a}.\label{control2b}
\end{align} \end{subequations} In this context, the optimal control problem consists of finding a solution of the state variables and control inputs for the previous equations \eqref{control2} given boundary conditions and minimizing the cost functional  
\[
{\mathcal J}(s^{(2,1)})=\int_{0}^{T} C\left(p^{a},\dot{p}^{a}, \sigma^{a}, \frac{\partial\mathcal{L}}{\partial p^{a}}-\frac{d}{dt}\frac{\partial\mathcal{L}}{\partial \dot{p}^{a}}-\frac{\partial\mathcal{L}}{\partial\sigma^{b}}(B_{ca}^{b}\dot{p}^{c}+C_{de}^{b}\sigma^{d}A_{a}^{e})\right)\, dt.
\] 
\textcolor{blue}{Necessary conditions for optimality in} the optimal control problem \textcolor{blue}{are}  characterized by the constrained second-order variational problem determined by the second-order Lagrangian 
\begin{equation}\label{LagOpt}
\widehat{\mathcal{L}}(s^{(2,1)})=C\left(p^{a},\dot{p}^{a}, \sigma^{a}, \frac{\partial\mathcal{L}}{\partial p^{a}}-\frac{d}{dt}\frac{\partial\mathcal{L}}{\partial \dot{p}^{a}}-\frac{\partial\mathcal{L}}{\partial\sigma^{b}}(B_{ca}^{b}\dot{p}^{c}+C_{de}^{b}\sigma^{d}A_{a}^{e})\right)
\end{equation}
 subject to the second-order constraints 
\begin{equation}\label{ConstOpt}
\chi^{\alpha}(s^{(2,1)})=\frac{d}{dt}\frac{\partial\mathcal{L}}{\partial\sigma^{\beta}}-\frac{\partial\mathcal{L}}{\partial\sigma^{\beta}}(C_{\gamma\beta}^{\delta}\sigma^{\gamma}-C_{\gamma\beta}^{\delta}A_{\epsilon}^{\gamma}\dot{p}^{\epsilon})
\end{equation}
 whose solutions satisfy the constrained second-order Lagrange-Poincar\'e equations for $\widetilde{\mathcal{L}}(s^{(2,1)},\lambda_{\alpha})=\widehat{\mathcal{L}}(s^{(2,1)})+\lambda_{\alpha}\chi^{\alpha}(s^{(2,1)})$ with $\lambda_{\alpha}\in\mathbb{R}^{m}$ the Lagrange multipliers. 

Those equations are in general given by a set of fourth order nonlinear ordinary differential equations which are very difficult to solve explicitly. Thus, constructing numerical methods is in order, a task for which the results in the previous sections must be implemented. 

\begin{remark}
It is well known that, under some mild regularity conditions, necessary conditions for optimality obtained through a constrained variational principle, are equivalent to the ones given by Pontryagin Maximum Principle (see \cite{Bloch}, section $7.3$, Theorem $7.3.3$ for the proof).

For higher-order systems, the same result can be proved. In particular, in \cite{CoPM}  for unconstrained higher-order mechanical system without symmetries the equivalence between higher-order Euler-Lagrange equations and higher-order Hamilton equations was shown. It would be interesting to study such equivalence for constrained systems and the relationship with necessary conditions for optimality in optimal control problems of underactuated mechanical systems. Such results were demonstrate for nonholonomic systems in \cite{BlCoGuMdD}, where the equivalence between conditions for optimal solutions obtained by the Pontryagin Maximum Principle and as a constrained variational problem for this particular class of constraints was established. Moreover, once such equivalence for constrained systems can be understood, by using the results of \cite{GHR2011} and \cite{tomoki} the relation between optimality conditions obtained by a constrained variational principle and the ones obtained by Pontryagin Maximum principle can be extended for the class of higher-order systems with symmetries studied in this work. \hfill$\diamond$

\end{remark}

Given discretizations of \eqref{LagOpt} and \eqref{ConstOpt}, denoted $\mathcal{L}_{d}$ and $\chi_{d}^{\alpha}$ respectively, defined on $3U\times 2\,G$, with local coordinates  $\tilde a_n=\lp p_{n-2},p_{n-1},p_{n},\widetilde g_{n-2},\widetilde g_{n-1}\rp$, $\widetilde g_i:= g_i^{-1}g_{i+1}A(p_i,p_{i+1})$, $2\leq n\leq N-2$, the associated discrete optimal control problem consist of obtaining the sequences $\lc p_n\rc_{0:N}$,  $\lc \tilde{g}_n\rc_{0:N}$ and $\lc \lambda_n\rc_{0:N}$ from the second-order constrained discrete Lagrange-Poincar\'e equations, i.e. \eqref{EqHigherOrder} for $k=2$. By Theorem \ref{HOTheorem}, the discrete constrained second-order Lagrange-Poincar\'e equations are given by

\begin{align}
0=&D_{1}\mathcal{L}_{d}(\tilde a_n)+D_{2}\mathcal{L}_{d}(\tilde a_{n-1})+D_{3}\mathcal{L}_{d}(\tilde a_{n-2})+T^{*}\hat{L}_{WA_1}(n)(D_{4}\mathcal{L}_{d}(\tilde a_n)+D_{5}\mathcal{L}_{d}(\tilde a_{n-1}))\nonumber\\
&+T^{*}\hat{L}_{WA_2}(n-1)(D_{4}\mathcal{L}_{d}(\tilde a_{n-1})+D_{5}\mathcal{L}_{d}(\tilde a_{n-2}))+\lambda_{\alpha}^{n}D_{1}\chi_{d}^{\alpha}(\tilde a_n)+\nonumber\\
&+\lambda_{\alpha}^{n-2}D_{3}\chi_{d}^{\alpha}(\tilde{a}_{n-2})+T^{*}\hat{L}_{WA_1}(n)(\lambda_{\alpha}^{n}D_{4}\chi_{d}^{\alpha}(\tilde a_n)+\lambda_{\alpha}^{n-1}D_{5}\chi_{d}^{\alpha}(\tilde a_{n-1}))\nonumber\\
&+T^{*}\hat{L}_{WA_2}(n-1)(\lambda_{\alpha}^{n-1}D_{4}\chi_{d}^{\alpha}(\tilde a_{n-1})+\lambda_{\alpha}^{n-2}D_{5}\chi_{d}^{\alpha}(\tilde a_{n-2}))+\lambda_{\alpha}^{n-1}D_{2}\chi_{d}^{\alpha}(\tilde a_{n-1}),\label{equation1example}\\
0=&M_{n}-\mbox{Ad}^*_{W_{n-1}}M_{n-1}+\lambda_{\alpha}^{n}\,\varepsilon^{\alpha}_{(n,4)}+\lambda_{\alpha}^{n-1}\,\varepsilon^{\alpha}_{(n,5)}\nonumber\\
&\quad\quad\quad\quad\quad\quad-\mbox{Ad}^*_{W_{n-1}}(\lambda_{\alpha}^{n-1}\,\varepsilon^{\alpha}_{(n-1,4)}+\lambda_{\alpha}^{n-2}\,\varepsilon^{\alpha}_{(n-1,5)}),\label{equation2example}\\
0=&\chi^{\alpha}_{d}(\tilde a_n), \quad 0=\chi^{\alpha}_d(\tilde a_{n-1}),\quad 0=\chi^{\alpha}_{d}(\tilde a_{n-2}),\label{constraintsexample}
\end{align} for $2\leq n\leq N-2$, and where 
\begin{align*}
M_n=&T^*_{W_{n}}R_{W_{n}^{-1}}(T^*_{W_{n}A_{n}}R_{A_{n}^{-1}}(D_{4}\mathcal{L}_{d}(\tilde a_n)+D_{5}\mathcal{L}_{d}(\tilde a_{n-1}))),\\
\varepsilon_{(n,4)}^{\alpha}=&T^*_{W_{n}}R_{W_{n}^{-1}}(T^*_{W_{n}A_{n}}R_{A_{n}^{-1}}D_{4}\chi_d^{\alpha}(\tilde a_{n})),\\
\varepsilon_{(n,5)}^{\alpha}=&T^*_{W_{n}}R_{W_{n}^{-1}}(T^*_{W_{n}A_{n}}R_{A_{n}^{-1}}D_{5}\chi_d^{\alpha}(\tilde a_{n-1})),\\
\varepsilon_{(n-1,4)}^{\alpha}=&T^*_{W_{n-1}}R_{W_{n-1}^{-1}}(T^*_{W_{n-1}A_{n-1}}R_{A_{n-1}^{-1}}D_{4}\chi_d^{\alpha}(\tilde a_{n-1})),\\
\varepsilon_{(n-1,5)}^{\alpha}=&T^*_{W_{n-1}}R_{W_{n-1}^{-1}}(T^*_{W_{n-1}A_{n-1}}R_{A_{n-1}^{-1}}D_{5}\chi_d^{\alpha}(\tilde a_{n-2})).
\end{align*}

By Proposition \ref{proposition2} the equations given above determine (locally) the flow map for the numerical method: they indicate how to  obtain $\tilde a_n$ and $\lambda^{n}$ given $\tilde a_{n-1}$, $\tilde a_{n-2}$, $\lambda^{n-1}$, $\lambda^{n-2}$ if the matrix \begin{equation*}\label{regu}
\begin{bmatrix}
 D_{13}\widetilde{\mathcal{L}}_{d}(\tilde a_n,\lambda^{n}) &  D_{5}\left(T^{*}\widehat{L}_{(WA_1)}(n)D_4\widetilde{\mathcal{L}}_{d}(\tilde a_n,\lambda^n)\right) &M_{(1,4)}(\tilde a_n) \\
D_{3}M_n(\tilde a_n) &  D_5M_n(\tilde a_n) &\varepsilon_{(n,4)}^{\alpha}(\tilde a_n)\\
(D_{3}\chi_{d}^{\alpha}(\tilde a_n))^{T} & (D_{5}\chi_{d}^{\alpha}(\tilde a_n))^{T} & 0 \\
\end{bmatrix}
\end{equation*} 
is non singular, where $M_{(1,4)}(\tilde{a}_{n}):=D_1\chi_d^{\alpha}(\tilde a_n)+T^{*}\hat L_{(WA_1)}(n)D_{4}\chi_d^{\alpha}(\tilde a_n)$.

\subsection{Examples}
\subsubsection{Optimal control of an electron in a magnetic field}\label{ParticleEx}

We study the optimal control problem for the linear momentum and charge of
an electron of mass $m$ in a given magnetic field (see \cite{Bloch} Section $3.9$). 

One of the motivations for constructing structure preserving variational integrators for this example is that the charge is a conserved quantity and our method, since it is variational, preserves the momentum map associated with a Lie group of symmetries.

Let $\mathcal{M}$ be a $3$ dimensional Riemannian manifold and $\pi:Q\ra \mathcal{M}$ be a circle
bundle (that is, $\mathbb{S}^{1}$ acts on $Q$ on the left and then
$\pi:Q\ra \mathcal{M}$ is a principal bundle where $\mathcal{M}=Q/\mathbb{S}^{1}$) with
respect to a left $SO(2)$ action.  We will use the isomorphism (as Lie group) of $SO(2)$ and $\mathbb{S}^{1}$ to make our analysis consistent with the theory.

 Let $\mathcal{A}:TQ\ra\mathfrak{so}(2)$ be a principal connection on $Q$
and consider the Lagrangian on $TQ$ given by
$$L(q,\dot{q})=\frac{m}{2}||T\pi(q,\dot{q})||_{M}^{2}+\frac{e}{c}||\mathcal{A}(q,\dot{q})||_{\mathfrak{so}(2)}-\phi(\pi(q)),$$
where $e$ is the charge of the electron, $c$ is the speed of light,
$||\cdot||_{\mathfrak{so}(2)}:\mathfrak{so}(2)\ra\R$ the
norm on $\mathfrak{so}(2)$, given by $\|\xi\|_{\mathfrak{so}(2)}=\langle\langle\xi,\xi\rangle\rangle^{1/2} = \sqrt{\hbox{tr}(\xi^{T}\xi)}$, for any $\xi\in\mathfrak{so}(2)$ where the inner product on $\mathfrak{so}(2)$ is given by $\langle\langle\xi,\xi\rangle\rangle=\hbox{tr}(\xi^{T}\xi)$. $\phi:\mathcal{M}\ra\R$ represents the potential energy and $\cdot$ denotes the left-action of $\mathbb{S}^{1}$ on $Q$. Note that in the absence of potential, $L$ is a Kaluza-Klein Lagrangian type (see \cite{CeMaRa} for instance).

The motivation for including  a potential function in our analysis is twofold. Firstly, it is inspired by possible further applications including static obstacles in the workspace. We  use $\phi$ as a artificial potential function (for instance a Coloumb potential) to avoid the obstacle. Secondly, it is motivated by use of this example in the theory of controlled Lagrangians and potential shaping for systems with breaking symmetries. Note that here $V$ is not invariant under the symmetry group (see \cite{Bloch} Section $4.7$) for more details.

Note also that $\pi(\theta\cdot q)=\pi(q)$ for all $q\in Q$ and
$\theta\in\mathbb{S}^{1}.$ Thus

\begin{align}
L(\theta\cdot(q,\dot{q}))&=\frac{m}{2}||T\pi(\theta\cdot(q,\dot{q}))||_{M}^{2}+\frac{e}{c}||\mathcal{A}(\theta\cdot(q,\dot{q}))||_{\mathfrak{so}(2)}-\phi(\pi(\theta\cdot q))\nonumber\\
&=\frac{m}{2}||T\pi(q,\dot{q})||_{M}^{2}+\frac{e}{c}||Ad_{\theta}\cdot \mathcal{A}(q,\dot{q})||_{\mathfrak{so(2)}}-\phi(\pi(q))\nonumber\\
&=\frac{m}{2}||T\pi(q,\dot{q})||_{M}^{2}+\frac{e}{c}||\mathcal{A}(q,\dot{q})||_{\mathfrak{so(2)}}-\phi(\pi(q))\nonumber\\
&=L(q,\dot{q})\nonumber
\end{align} where Ad$_{\theta}=$Id$_{\mathfrak{so}(2)}$ because $SO(2)$ is Abelian. That is, $L$ is $SO(2)$-invariant and
we may perform Lagrange-Poincar\'e reduction by symmetries to get the equations of motion on
the principal bundle $TQ/SO(2).$

Fixing the connection $\mathcal{A}$ on $Q$, we can use the principal
connection $\mathcal{A}$ to get an isomorphism $\alpha_{\mathcal{A}}:TQ/SO(2)\ra
T\mathcal{M}\oplus\widetilde{\mathfrak{so}}(2)$ which permits us to define the reduced Lagrangian

$$\mathcal{L}(x,\dot{x},\xi)=\frac{m}{2}||\dot{x}||_{M}^{2}+\frac{e}{c}||\xi||_{\mathfrak{so}(2)}-\phi(x).$$

For the reduced Lagrangian $\ell$, the dynamics is determined by the Lagrange-Poincar\'e equations \eqref{LagPoincareGeneral}, in this particular case
\begin{align*}
m\frac{D\dot{x}^{\flat}}{Dt}+\mathbf{d}\phi&=\langle\mu,\widetilde{\mathcal{B}}(\dot{x}(t),\cdot)\rangle\nonumber\\
\frac{D}{Dt}\mu&=0\nonumber
\end{align*} where $\mu=\frac{\partial\mathcal{L}}{\partial\xi}$ is the charge of the particle. Here, $\widetilde{\mathcal{B}}:T\mathcal{M}\wedge T\mathcal{M}\ra\mathfrak{so}(2)$ is the reduced
curvature tensor associated with the connection form $\mathcal{A}$, $\mathbf{d}$ is the exterior differential and $\flat:\mathfrak{g}\to\mathfrak{g}^{*}$ is the associated isomorphisms to the inner product defined by the metric (see \cite{Bloch} and \cite{bookBullo} for instance). Note that this equation corresponds with Wong's equations \cite{CeMaRa}.

In the case where $Q=\R^{3}\times\mathbb{S}^{1}$ the Lagrangian is

$$L(x,\dot{x},\theta,\dot{\theta})=\frac{m}{2}\dot{x}^{2}+\frac{e}{c}(A(x,\dot{x})\cdot\dot{x})-\phi(x).$$
In this case, we have that $TQ/SO(2)\simeq\R^{3}\times\R$ where Ad$Q=\R$
and the reduced Lagrangian is
$$\mathcal{L}(x,\dot{x},\xi)=\frac{m}{2}\dot{x}^2+\frac{e}{c}\xi-\phi(x).$$ The
above equations reduce to type of Lorentz force law describing the
motion of a charged particle of mass $m$ in a magnetic field under the influence of a potential function
$$m\ddot{x}+\nabla\phi(x)=\frac{e}{c}(\dot{x}\times\overrightarrow{B}),\quad
\dot{\mu}=0,$$ where $\mu=\frac{\partial\ell}{\partial\xi}=\frac{e}{c}$ and $\overrightarrow{B}=(B_x, B_y, B_z)\in\mathfrak{X}(\R^{3}).$

Next, we introduce controls in our picture. Let $U\subset\R^{3}$,  where $u=(u_1,u_2,u_3)\in U$ are the control inputs. Then, given $u(t)\in U$, the controlled decoupled Lagrange-Poincar\'e system  \eqref{control1} is given by  \begin{align}
m\frac{D\dot{x}^{\flat}}{Dt}+\mathbf{d}\phi-\langle \mu,\widetilde{\mathcal{B}}(\dot{x}(t),\cdot)\rangle&=u(t),\nonumber\\
\frac{D}{Dt}\mu&=0\nonumber.
\end{align} If $Q=\R^{3}\times\mathbb{S}^{1}$ then the above system becomes the controlled decoupled Lagrange-Poincar\'e system describing the controlled dynamics of a charged particle of mass $m$ in a magnetic field under the influence of a potential function:

\begin{align}
m\ddot{x}+\nabla\phi(x)-\frac{e}{c}(\dot{x}\times\overrightarrow{B})&=u(t)\nonumber\\
\dot{\mu}&=0.\nonumber
\end{align}

The optimal control problem consists
of finding trajectories of the state variables and controls inputs,
satisfying the previous equations subject to given initial and final conditions
and minimizing the cost functional,
$$\min_{(x, \dot{x}, \xi,u)}\int_{0}^{T}C(x, \dot{x}, \xi,u)dt=\min_{(x, \dot{x}, \xi,u)}\frac{1}{2}\int_{0}^{T}||u||^{2}\,dt$$ where the norm $||\cdot||$ represents the Euclidean norm on $\mathbb{R}^{3}$.

This optimal control problem is equivalent to solving the following
constrained second-order variational problem given
by  
\begin{equation}\label{LagrangianParticle}
\min_{(x, \dot{x},\ddot{x}, \xi,\dot{\xi})}\widehat{\mathcal{L}}(x, \dot{x},\ddot{x}, \xi,\dot{\xi})=\frac{1}{2}\left|\left|m\ddot{x}+\nabla\phi(x)-\frac{e}{c}(\dot{x}\times\overrightarrow{B})\right|\right|^{2},
\end{equation} subject to the constraint
$\chi(x, \dot{x},\ddot{x}, \xi,\dot{\xi})=\frac{e}{c}$ arrising from $\dot{\mu}=0$, with $\widehat{\mathcal{L}}:3\mathbb{R}^{3}\times 2\mathbb{R}\to\mathbb{R}$ and $\chi:3\mathbb{R}^{3}\times 2\mathbb{R}\to\mathbb{R}$ (note that $TQ/SO(2)\simeq\R^{3}\times\R$ where Ad$Q=\R$). 

For a simple exposition of the resulting equations describing necessary conditions for optimality in the optimal control problem, we restrict our analysis to the particular case when the magnetic field is aligned with the $x_3$-direction and orthogonal to the $x_1-x_2$ plane, that is, $\overrightarrow{B}=(0,0,B_z)$ with $B_z$ constant, and the potential field is quadratic $\phi=(x_1^2+x_2^2+x_3^2)$.

The constrained second-order Lagrange-Poincar\'e equations  are 
\begin{subequations}\label{Particle}
 \begin{align}
x_1^{(iv)}&=2\omega\dddot{x}_2+\ddot{x}_{1}\left(\omega^2-\frac{4}{m}\right)+\frac{4\omega}{m}\dot{x}_2-\frac{4x_1}{m^2},\label{Particle:a}\\
x_2^{(iv)}&=-2\omega\dddot{x}_{1}+\ddot{x}_{2}\left(\omega^2-\frac{4}{m}\right)-\frac{4\omega}{m}\dot{x}_{1}-\frac{4}{m^2}x_2,\label{Particle:b}\\
x_3^{(iv)}&=-\frac{4\ddot{x}_3}{m}-\frac{4x_3}{m^2},\label{Particle:c}
\end{align}
\end{subequations} 
where $\omega=\frac{eB_z}{mc}$, $\lambda(t)$ is constant and $\xi(t)=\frac{e}{c}$. $\xi(t)$ comes from the the constraint given by preservation of the charge and $\lambda(t)$  is obtained from  $\frac{e}{c}\dot{\lambda}=0$, the Lagrange-Poincar\'e equation arising from $\xi(t)$ (note that we obtain the same result for the multiplier as in \cite{Bloch} Section $7.5$.)

In terms of the discretization of this system as presented in Section \ref{Discretization}, we need to define the discrete connection \eqref{DiscConnec}, but given that the bundle is trivial, the connection vanishes. Denoting by $(x_n,\xi_n)=(x_n^1,x_n^2,x_n^3,\xi_n,\xi_{n+1})$, the discrete second order Lagrangian for the reduced optimal control problem corresponding to \eqref{LagrangianParticle}, is given by

\begin{align*}
\widehat{\mathcal{L}}_d(x_n,x_{n+1},x_{n+2},\xi_n)&=\frac{h}{2}\Big{|}\Big{|}\frac{x_{n+2}-2x_{n+1}+x_n}{h^2}+\nabla\phi(x_n)\\&\qquad\qquad\qquad\qquad\qquad-\frac{e}{c}\frac{x_{n+1}-x_n}{h}\times \overrightarrow{B}(x_n)\Big{|}\Big{|}^2,\\
\chi_d(x_n,x_{n+1},x_{n+2},\xi_n)&=\frac{e}{c},
\end{align*} 
where $\overrightarrow{B}(x_n):=(B_1(x_n^{1}), B_2(x_n^{2}), B_3(x_n^{3}))$.

By Theorem \ref{HOTheorem} the discrete second-order constrained Lagrange-Poincare equations giving rise to the integrator which approximates the necessary conditions for optimality in the optimal control problem are given by
\begin{subequations}\label{DiscParticle}
\begin{align}
\frac{x_{n+2}^1-4x_{n+1}^1+6x_{n}^1-4x_{n-1}^1+x_{n-2}^1}{h^4}&=2\omega\frac{x_{n+1}^2-3x_n^2+3x_{n-1}^2-x_{n-2}^2}{h^3}\label{DiscParticle:a}\\
&\quad\quad+\lp\omega^2-\frac{4}{m}\rp\frac{x_{n+1}^1-2x_n^1+x_{n-1}^1}{h^2}\nonumber\\
&\quad\quad\quad\quad\quad\quad\quad+\frac{4\omega}{m}\frac{x_n^2-x_{n-1}^2}{h}-\frac{4}{m^2}x_n^1,\nonumber\\
\frac{x_{n+2}^2-4x_{n+1}^2+6x_{n}^2-4x_{n-1}^2+x_{n-2}^2}{h^4}&=-2\omega\frac{x_{n+1}^1-3x_n^1+3x_{n-1}^1-x_{n-2}^1}{h^3}\label{DiscParticle:b}\\
&\quad\quad\quad+\lp\omega^2-\frac{4}{m}\rp\frac{x_{n+1}^2-2x_n^2+x_{n-1}^2}{h^2}\nonumber\\
&\quad\quad\quad\quad\quad\quad\quad -\frac{4\omega}{m}\frac{x_n^1-x_{n-1}^1}{h}-\frac{4}{m^2}x_n^2,\nonumber
\\
\frac{x_{n+2}^3-4x_{n+1}^3+6x_{n}^3-4x_{n-1}^3+x_{n-2}^3}{h^4}&=-\frac{4}{m}\frac{x_{n+2}^3-2x_n^3+x_{n-1}^3}{h^2}-\frac{4}{m^2}x_n^3,\label{DiscParticle:c}
\end{align}
\end{subequations} 
together with $\xi_{n}=\xi_{n-1}=\xi_{n-2}=\frac{e}{c}$,  and  $\lambda_{n}=\lambda_{n-1}$ for $n=2,\ldots,N-2$. We observe that \eqref{DiscParticle:a}, \eqref{DiscParticle:b} and \eqref{DiscParticle:c} are a discretization in finite differences of \eqref{Particle:a}, \eqref{Particle:b} and \eqref{Particle:c}, respectively.\footnote{Considering the forward difference $\frac{x_{n+1}-x_n}{h}$ as a first order discretization of the velocity $\dot x$,  it is straightforward to check that
\begin{align*}
\frac{x_{n+2}-2x_{n+1}+x_{n}}{h^2},&\\
\frac{x_{n+3}-3x_{n+2}+3x_{n+1}-x_n}{h^3},&\\
\frac{x_{n+4}-4x_{n+3}+6x_{n+2}-4x_{n+1}+x_{n}}{h^4}&,\\
\end{align*}
are discretization of $\ddot x$, $\dddot x$ and $x^{(iv)}$, respectively. The shift of the $n$ index present in equations \eqref{DiscParticle} comes from the particular expression of the discrete constrained HO Lagrange-Poincar\'e equations provided in Theorem \ref{HOTheorem}.}

\subsubsection{Energy minimum control of two coupled rigid bodies:}

We consider a discretization of the energy minimum control for the motion planning of an underactuated system composed by two planar rigid bodies attached at their center of mass and moving freely in the plane, also known in the literature as Elroy's beanie (see \cite{lewis}, \cite{Ostrowski} for details) which is  an example of a dynamical system with a non-Abelian Lie group of symmetries. 

The configuration space is $Q=SE(2)\times \mathbb{S}^1$ with local coordinates  denoted by $(x,y,\theta,\psi)$. The Lagrangian function $L:TQ\Flder\R$ is given by
\begin{equation}\label{LagElBean}
L(x,y,\theta,\psi,\dot x,\dot y,\dot\theta,\dot\psi)=\frac{1}{2}m\,(\dot x^2+\dot y^2)+\frac{1}{2}\,I_1\,\dot\theta^2+\frac{1}{2}\,I_2\,(\dot\theta+\dot\psi)^2-V(\psi),
\end{equation}
where $m$ denotes the mass of the system, $I_1$ and $I_2$ are the inertias of the first and the second body,
respectively, and $V$ is the potential energy. Note that the system is invariant under $SE(2)$. After choosing a decomposition determined by the metric on $Q$ which describes the kinetic energy of the Lagrangian \eqref{LagElBean}  (see \cite{lewis} and \cite{Ostrowski} for details) one can fix a connection $\cone:T(SE(2)\times  \mathbb{S}^1)\Flder \se$, with local expression
\begin{equation}\label{ContConneExam}
\cone=\begin{bmatrix}
1&0&y&\frac{I_2}{I_1+I_2}\,y\\
0&1&-x&-\frac{I_2}{I_1+I_2}\,x\\
0&0&1&\frac{I_2}{I_1+I_2}
\end{bmatrix}, 
\end{equation}
and vanishing curvature. 

Consider the base of $\mathfrak{se}(2)$ denoted by $\{\bar{e}_{a}\}$ with $a=1,2,3$ and given by 
 $$ \bar{e}_1=\begin{bmatrix}
0&0&1\\
0&0&0\\
0&0&0
\end{bmatrix},\,\,\bar{e}_2=\begin{bmatrix}
0&0&0\\
0&0&1\\
0&0&0
\end{bmatrix},\,\, \bar{e}_3=\begin{bmatrix}
0&1&0\\
-1&0&0\\
0&0&0
\end{bmatrix}.$$ In terms of this basis, $\xi\in\mathfrak{se}(2)$ can be written as $\xi=\xi^{1}\bar{e}_1+\xi^{2}\bar{e}_{2}+\xi^{3}\bar{e}_{3}$ with $\xi^{1}=\cos\theta\dot{x}+\sin\theta\dot{y}$, $\xi^{2}=\sin\theta\dot{x}+\cos\theta\dot{y}$ and $\xi^{3}=-\dot{\theta}-\frac{I_2}{I_1+I_2}\dot{\psi}$. Moreover, since $ 
[\bar{e}_1, \bar{e}_2]=0$, $[\bar{e}_1,\bar{e}_3]=\bar{e}_2$ and $[\bar{e}_2, \bar{e}_3]=-\bar{e}_1$, then the non-vanishing constant structures of the Lie algebra $\mathfrak{se}(2)$ are $C_{13}^{2}=C_{32}^{1}=1$ and $C_{31}^{2}=C_{23}^{1}=-1$.

The isomorphism \eqref{IsomorphismCont}, $\alpha_{\mathcal{A}}:T(SE(2)\times \mathbb{S}^{1})/SE(2)\to T\mathbb{S}^{1}\times\widetilde{\mathfrak{se}}(2)$ is  $$\alpha_{\mathcal{A}}(\psi,\dot{\psi},\xi)=(\psi,\dot{\psi},\xi+A_{e}(\psi)\dot{\psi}),$$ with $A_{e}(\psi)=(0, 0, \frac{I_{2}}{I_1+I_2})^{T}$ ($A_{e}:U\subset \mathbb{S}^{1} \to\mathfrak{se}(2)$ is a $1$-form determined by $A_{e}(\psi)\dot{\psi}=\mathcal{A}(\psi,e,\dot{\psi},0)$). $A_e$ is locally prescribed by the coefficients $A_1^{1}=A_{1}^{2}=0$ and $A_1^{3}=\frac{I_2}{I_1+I_2}$.

The reduced Lagragian (see \cite{lewis, Ostrowski} for details) $\mathcal{L}:T\mathbb{S}^1\oplus\widetilde{\mathfrak{se}}(2)\Flder\R$ is given by
\begin{equation}\label{RedContLag}
\mathcal{L}(\psi,\dot\psi,\Omega)=\frac{1}{2}m(\Omega_1^2+\Omega_2^2)+\frac{1}{2}(I_1+I_2)\,\Omega_3^2+\frac{1}{2}\frac{I_1I_2}{I_1+I_2}\dot\psi^2-V(\psi),
\end{equation}
where $(\psi,\dot\psi)$ are local coordinates for $T\mathbb{S}^1$ and $\Omega$ for $\widetilde{\mathfrak{se}}(2)$, such that $\Omega^1=\xi^1$, $\Omega^2=\xi^2$ and $\Omega^3=\xi^3+\frac{I_2}{I_1+I_2}\dot\psi$, for $\xi\in\se$, and consequently we observe that 
\begin{subequations}\label{Omegas}
\begin{align}
\Omega^1=&\cos{\theta}\,\dot x+\sin{\theta}\,\dot y,\\
\Omega^2=&-\sin{\theta}\,\dot x+\cos{\theta}\,\dot y,\\
\Omega^3=&-\dot\theta.
\end{align}
\end{subequations}

The local Lagrange-Poincar\'e equations \eqref{LagPoinLocal} in the $(\psi,\dot\psi,\Omega)$ coordinates read
\begin{subequations}\label{RedEoM}
\begin{align}
\dot\Omega^1=&\Omega^2\lp\,\Omega^3-\frac{I_2}{I_1+I_2}\,\dot\psi\rp,\label{RedEoMa}\\
\dot\Omega^2=&-\Omega^1\lp\,\Omega^3-\frac{I_2}{I_1+I_2}\,\dot\psi\rp,\label{RedEoMb}\\
\dot\Omega^3=&0,\label{RedEoMc}\\
\frac{I_1I_2}{I_1+I_2}\,\ddot\psi=&-\frac{\partial V}{\partial\psi}.\label{RedEoMd}
\end{align}
\end{subequations}

Now, we introduce a control input in the equation corresponding to $\mathbb{S}^{1}$, i.e. \eqref{RedEoMd}, namely 
\begin{equation}\label{controlled}
\frac{I_1I_2}{I_1+I_2}\,\ddot\psi=-\der_{\psi}V(\psi)+u.
\end{equation}
As discussed in the previous subsection, the optimal control problem consists of finding trajectories of the state variables and control inputs, satisfying the  equations \eqref{RedEoMa}, \eqref{RedEoMb}, \eqref{RedEoMc} and \eqref{controlled}, subject to boundary conditions and minimizing the cost functional $\int_0^TC(\psi,\dot\psi,\Omega,u)dt.$ In particular we are interested in energy-minimum problems, where the cost function is of the form $\displaystyle{
C(\psi,\dot\psi,\Omega,u)=\frac{1}{2}u^2.}$
This optimal control problem is equivalent to solving the constrained second-order variational problem defined by the Lagrangian $\widehat{ \mathcal{L}}:T^{(2)}\mathbb{S}^{1}\oplus2\widetilde{\mathfrak{se}}(2)\Flder\R$ and the constraints $\chi^{\alpha}:T^{(2)}\mathbb{S}^{1}\oplus2\widetilde{\mathfrak{se}}(2)\Flder\R$, $\alpha=1,2,3$, defined by

$\widehat{ \mathcal{L}}(\gamma)=\frac{1}{2}\lp\frac{I_1I_2}{I_1+I_2}\,\ddot\psi+\der_{\psi}V(\psi)\rp^2$, $\displaystyle{ 
\chi^1(\gamma)=\dot\Omega^1-\Omega^2\Omega^3+\frac{I_2}{I_1+I_2}\,\dot\psi\,\Omega^2}$, $\chi^2(\gamma)=\dot\Omega^2+\Omega^1\Omega^3-\frac{I_2}{I_1+I_2}\,\dot\psi\,\Omega^1$ and 
$\chi^3(\gamma)=\dot\Omega^3$, where $\gamma=(\psi,\dot\psi,\ddot\psi,\Omega,\dot\Omega)$, $(\psi,\dot\psi,\ddot\psi)$ are local coordinates for $T^{(2)}\mathbb{S}^{1}$ and $(\Omega, \dot\Omega)$ for $2\widetilde{\mathfrak{se}}(2)$. Next, define the augmented Lagrangian $\widetilde{\mathcal{L}}(\psi,\dot\psi,\ddot\psi,\Omega,\dot\Omega)=\widehat{\mathcal{L}}(\psi,\dot\psi,\ddot\psi,\Omega,\dot\Omega)+\lambda_{\alpha}\chi^{\alpha}(\psi,\dot\psi,\ddot\psi,\Omega,\dot\Omega)$, that is, \begin{align}\label{AugmSecOrdeLag}
\widetilde{\mathcal{L}}(\psi,\dot\psi,\ddot\psi,\Omega,\dot\Omega,\lambda_1,\lambda_2,\lambda_3)=&\frac{1}{2}\lp\frac{I_1I_2}{I_1+I_2}\,\ddot\psi+\der_{\psi}V(\psi)\rp^2
\\&+\lambda_1\lp\dot\Omega_1-\Omega^2\Omega^3+\frac{I_2}{I_1+I_2}\,\dot\psi\,\Omega^2\rp\nonumber\\
&\quad\quad\quad\quad+\lambda_2\lp\dot\Omega^2+\Omega^1\Omega^3-\frac{I_2}{I_1+I_2}\,\dot\psi\,\Omega^1\rp+\lambda_3\dot\Omega^3.
\end{align}

The constrained second-order Lagrange-Poincar\'e equations determining necessary conditions for the optimal control problem (given in Remark \ref{remarklocal2}) are the following fourth-order nonlinear system of equations:

\begin{subequations}\label{BeanieCont}
\begin{align}
\frac{I_1I_2}{I_1+I_2}\psi^{(iv)}=&\lambda_1\left(\Omega^1\Omega^3+\dot{\Omega}^{2}-\Omega^{1}\dot{\psi}\frac{I_2}{I_1+I_2}\right)+\lambda_2\left(\Omega^2\Omega^3-\dot{\Omega}^{1}-\Omega^{2}\dot{\psi}\frac{I_2}{I_1+I_2}\right)\\
\nonumber
&-\frac{d^{2}}{dt^{2}}\left(\frac{\partial V}{\partial\psi}\right)-\frac{(I_1+I_2)}{I_2}\frac{\partial^2 V}{\partial\psi^2}\left(\frac{I_1I_2}{I_1+I_2}\ddot{\psi}+\frac{\partial V}{\partial\psi}\right),\\
\ddot{\lambda}_{1}=&\lambda_2\left(\dot{\Omega}^{3}-\frac{I_2\ddot{\psi}}{I_1+I_2}\right)+\frac{\lambda_1I_2}{I_1+I_2}\left(2\dot{\psi}\Omega^3-(\Omega^3)^2-\frac{\dot{\psi}I_2}{I_1+I_2}\right),\\
\ddot{\lambda}_2=&\lambda_1\left(\frac{\ddot{\psi}I_2}{I_1+I_2}-\dot{\Omega}^{3}\right)+\frac{\dot{\lambda}_1\dot{\psi}I_2(1-\lambda_2)}{I_1+I_2}-\lambda_2\left((\Omega^3)^2-\frac{I_2}{I_1+I_2}\dot{\psi}\dot{\Omega}^3\right)\\\nonumber
&\quad\quad\quad\quad\quad\quad\quad\quad\quad\quad\quad\quad\quad\quad\quad\quad\quad+\frac{\lambda_2^2\dot{\psi}}{I_1+I_2}\left(\Omega^3-\frac{I_2\dot{\psi}}{I_1+I_2}\right),\\
\ddot{\lambda}_{3}=&\Omega^2(\dot{\lambda}_2-2\dot{\lambda}_1)-\lambda_1\dot{\Omega}^2+\lambda_1\dot{\Omega}^{1}+\dot{\lambda}_{2}\Omega^1+(\lambda_1\Omega^1\\&+\lambda_2\Omega^2)\left(\Omega_3-\frac{I_2}{I_1+I_2}\dot{\psi}\right)\\
\dot\Omega^1=&\,\,\,\,\,\,\,\Omega^2\Omega^3-\frac{I_2}{I_1+I_2}\,\dot\psi\,\Omega^2,\\
\dot\Omega^2=&-\Omega^1\Omega^3+\frac{I_2}{I_1+I_2}\,\dot\psi\,\Omega^1,\\
\dot\Omega^3=&\,\,\,\,\,\,\,0,
\end{align}
\end{subequations}

In terms of the discretization of this system as presented in Section \ref{Discretization}, we need to define the discrete connection \eqref{DiscConnec} $\dcone:(SE(2)\times \mathbb{S}^{1})\times (SE(2)\times \mathbb{S}^{1})\Flder SE(2)$, which should satisfy $\dcone((g_n,\psi_n),(g_{n+1},\psi_{n+1}))=g_{n+1}A(\psi_n,\psi_{n+1})g_n^{-1}$ according to \eqref{DiscConeTrivialChart} for $g_n,g_{n+1}\in SE(2)$ and $\psi_n,\psi_{n+1}\in U\subset (SE(2)\times \mathbb{S}^{1})/SE(2)\cong \mathbb{S}^{1}$ . The local expression of the discrete connection is given by
\begin{equation}\label{LocalConne}
A(\psi_n,\psi_{n+1})=\begin{bmatrix}
\cos{\lp\frac{I_2}{I_1+I_2}\Delta\psi_n\rp}&-\sin{\lp\frac{I_2}{I_1+I_2}\Delta\psi_n\rp}&0\\
\sin{\lp\frac{I_2}{I_1+I_2}\Delta\psi_n\rp}&\cos{\lp\frac{I_2}{I_1+I_2}\Delta\psi_n\rp}&0\\
0&0&1
\end{bmatrix},
\end{equation}
with $\Delta\psi_n=\psi_{n+1}-\psi_n$. We denote $A_n=A(\psi_n,\psi_{n+1}).$ The reduced discrete Lagrangian $\mathcal{L}_d:\mathbb{S}^{1}\times \mathbb{S}^{1}\times  \widetilde{SE}(2)\Flder\R$ is locally defined by the coordinates $(\psi_n,\psi_{n+1},\widetilde g_n)$, where $\widetilde g_n=W_n\,A_n$, with $W_n=g_n^{-1}g_{n+1}$. According to  \eqref{LocalConne}
\[
\widetilde g_n =\begin{bmatrix}
\cos{\Delta\varphi_n} &-\sin{\Delta\varphi_n}& \cos{\theta_n}\Delta x_n+\sin{\theta_n}\Delta y_n\\
\sin{\Delta\varphi_n} &\cos{\Delta\varphi_n}& -\sin{\theta_n}\Delta x_n+\cos{\theta_n}\Delta y_n\\
0&0&1
\end{bmatrix},
\]
where $\Delta\varphi_n=\Delta\theta_n+\frac{I_2}{I_1+I_2}\Delta\psi_n$, $\Delta\theta_n=\theta_{n+1}-\theta_n$, $\Delta x_n=x_{n+1}-x_n$ and $\Delta y_n=y_{n+1}-y_n$. Establishing
\begin{align*}
\Omega_1^n=&\,\,\,\,\,\,\,\,\cos{\theta_n}\,(\Delta x_n/h)+\sin{\theta_n}\,(\Delta y_n/h),\\
\Omega_2^n=&\,\,-\sin{\theta_n}\,(\Delta x_n/h)+\cos{\theta_n}\,(\Delta y_n/h),\\
\Omega_3^n=&\,\,-(\Delta\theta_n/h)
\end{align*}
which represent a discretization of \eqref{Omegas} where $h$ is the time-step of the integrator, we obtain that
\begin{equation}\label{tildeg}
\widetilde g_n=\begin{bmatrix}
\cos{\lp h(\Omega_3^n-\frac{I_2}{I_1+I_2}\frac{\Delta\psi_n}{h})\rp}&\sin{\lp h(\Omega_3^n-\frac{I_2}{I_1+I_2}\frac{\Delta\psi_n}{h})\rp}&h\,\Omega_1^n\\
-\sin{\lp h(\Omega_3^n-\frac{I_2}{I_1+I_2}\frac{\Delta\psi_n}{h})\rp}&\cos{\lp h(\Omega_3^n-\frac{I_2}{I_1+I_2}\frac{\Delta\psi_n}{h})\rp}&h\,\Omega_2^n\\
0&0&1
\end{bmatrix},
\end{equation}
and it follows that $\widetilde g_n$ is completely determined by $(\Omega_1^n,\Omega_2^n,\Omega_3^n)$ after fixing $\Delta\psi_n$. Therefore, the discrete Lagrangian $\mathcal{L}_d:\mathbb{S}^{1}\times \mathbb{S}^{1}\times  \widetilde{SE}(2)\Flder\R$ is given by
\begin{equation}\label{DiscLagrangian}
\mathcal{L}_d(a_n)=\frac{1}{2}mh((\Omega_1^n)^2+(\Omega_2^n)^2)+\frac{1}{2}(I_1+I_2)h\,(\Omega_3^n)^2+\frac{1}{2}\frac{I_1I_2}{I_1+I_2}\frac{(\Delta\psi_n)^2}{h}-hV(\psi_{n+1}),
\end{equation} where $a_n=(\psi_n,\psi_{n+1},\Omega_1^n,\Omega_2^n,\Omega_3^n)$.

Considering a discretization of the action integral $\int_0^h\mathcal{L}\,dt$ determined by \eqref{RedContLag}  with  local truncation error of first order, the discrete equations of motion \eqref{DEoM} read 
\begin{subequations}\label{Integrator1}
\begin{align}
\Omega_1^n&=\Omega_1^{n-1}+h\Omega^{n-1}_2\lp\Omega_3^{n-1}-\frac{I_2}{I_1+I_2}\frac{\Delta\psi_{n-1}}{h}\rp,\label{Integrator1a}\\
\Omega_2^n&=\Omega_2^{n-1}-h\Omega^{n-1}_1\lp\Omega_3^{n-1}-\frac{I_2}{I_1+I_2}\frac{\Delta\psi_{n-1}}{h}\rp,\label{Integrator1b}\\
\Omega_3^n&=\Omega_3^{n-1},\label{Integrator1c}\\
\frac{I_1I_2}{I_1+I_2}\frac{\psi_{n+1}-2\psi_n+\psi_{n-1}}{h}&=-h\,\der_{\psi} V(\psi_n)\label{Integrator1d},
\end{align}
\end{subequations}
for the discrete Lagrangian \eqref{DiscLagrangian}, where we we have used the expressions for the local discrete connection \eqref{LocalConne}, $\widetilde g_n$ \eqref{tildeg}, and we neglected higher-order terms of the time step $O(h^2)$ (equations \eqref{Integrator1a},\eqref{Integrator1b} and \eqref{Integrator1c} follow from the second equation in \eqref{DEoM}, while \eqref{Integrator1d} follows from the first one). It is easy to check that \eqref{Integrator1}  is a discretization  in finite differences of \eqref{RedEoM}, (see footnote $^1$).

\begin{remark}
A different discretization of the potential $V$ in \eqref{DiscLagrangian}, for instance $$\displaystyle{h\,V\lp\frac{\psi_{n+1}+\psi_n}{2}\rp} \hbox{ or } \displaystyle{\frac{h}{2}V(\psi_{n+1})+\frac{h}{2}V(\psi_{n})},$$  would lead to a second-order discretization of \eqref{Integrator1d} with respect to \eqref{RedEoMd}. However, the local truncation error of \eqref{Integrator1} with respect to \eqref{RedEoM} does not change, since the order of the $\Omega$ equations remains the same. It seems strange to use asymmetric, $O(h)$ approximations when this is not necessary as we used in the previous example. This is a phenomenon due to the ``decoupling'' of the variables $\psi$ and $\Omega$ in equation \eqref{Integrator1d}, which allows to enhance its local truncation error via the appropriate discretization of the discrete Lagrangian. However, the overall $L_d$ is $O(h)$, and therefore one expects the same order of the integrator.

\hfill$\diamond$
\end{remark}

The discrete second-order augmented Lagrangian $\tilde{\mathcal{L}}_d:\mathbb{S}^{1}\times \mathbb{S}^{1}\times \mathbb{S}^{1}\times 2\widetilde {SE}(2)\times\R^3\Flder\R$ for the constrained higher-order variational problem is given by
\begin{equation}\label{AugmSecOrdeDiscLag}
\begin{split}
\tilde{\mathcal{L}}_d(\tilde{a}_{n},\lambda^n)=&
\frac{h}{2}\lp\frac{I_1I_2}{I_1+I_2}\,\frac{\psi_{n+2}-2\psi_{n+1}+\psi_{n}}{h^2}+\der_{\psi}V(\psi_{n+1})\rp^2\\
&\,\,+\lambda_3^n\,\Delta\Omega_3^n+\lambda_1^n\lp\Delta\Omega_1^n-h\Omega_2^n\Omega_3^n+\frac{I_2}{I_1+I_2}\,\Delta\psi_n\,\Omega_2^n\rp\\
&\quad\quad\quad\quad\quad\quad\quad+\lambda_2^n\lp\Delta\Omega_2^n+h\Omega_1^n\Omega_3^n-\frac{I_2}{I_1+I_2}\,\Delta\psi_n\,\Omega_1^n\rp,
\end{split}
\end{equation}
with $\tilde{a}_{n}=(\psi_n,\psi_{n+1},\psi_{n+2},\Omega^n,\Omega^{n+1})$, $\Omega^n=(\Omega_1^n,\Omega_2^n, \Omega_3^n)$, $\lambda^n=(\lambda_1^n,\lambda_2^n,\lambda_3^n)$, and with  $\Delta\Omega_i^n=\Omega_i^{n+1}-\Omega_i^{n}$. Here, to define $\widetilde{\mathcal{L}}_{d}$ we used
\begin{align}
\widehat{\mathcal{L}}_{d}(\tilde{a}_n)=&\frac{h}{2}\lp\frac{I_1I_2}{I_1+I_2}\,\frac{\psi_{n+2}-2\psi_{n+1}+\psi_{n}}{h^2}+\der_{\psi}V(\psi_{n+1})\rp^2\label{lagrangiandiscreteexamle}\\
\chi_d^{1}(\tilde{a}_{n})=&\Delta\Omega_1^n-h\Omega_2^n\Omega_3^n+\frac{I_2}{I_1+I_2}\,\Delta\psi_n\,\Omega_2^n\label{chi1}\\
\chi_d^{2}(\tilde{a}_{n})=&\Delta\Omega_2^n+h\Omega_1^n\Omega_3^n-\frac{I_2}{I_1+I_2}\,\Delta\psi_n\,\Omega_1^n\label{chi2}\\
\chi_d^{3}(\tilde{a}_{n})=&\Delta\Omega_3^n\label{chi3}
\end{align}

The discrete constrained second-order Lagrange-Poincare equations giving rise to the variational integrator to approximate the necessary conditions for optimality in the optimal control problem are given by equations \eqref{equation1example}, \eqref{equation2example}, \eqref{constraintsexample} applied to the discrete second-order augmented Lagrangian $\widetilde{\mathcal{L}}_d$ \eqref{AugmSecOrdeDiscLag} where the partial derivatives of $\widehat{\mathcal{L}}_d$ and $\chi_d^{\alpha}$, $\alpha=1,2,3$ follows easily from equations \eqref{lagrangiandiscreteexamle}-\eqref{chi3} and are understood as row vectors. The operators $T^{*}\hat{L}_{(WA_1)}$ and $T^{*}\hat{L}_{(WA_2)}$ can be computed using the tangent lift of left translations as in equation \eqref{Aoperator}, $M_n=[0,\, 0,\, 0]$, and the quantities $\epsilon^{\alpha}_{(n,4)}$, $\epsilon^{\alpha}_{(n,5)}$, $\epsilon^{\alpha}_{(n-1,4)}$, $\epsilon^{\alpha}_{(n-1,5)}$ are given as follow 

\begin{align*}
\epsilon^{1}_{(n,4)}=&[-\cos(h\vartheta_n)+h\vartheta_n\sin(h\vartheta_n),\,\sin(h\vartheta_n)+h\vartheta_n\cos(h\vartheta_n),\,h\Omega_1^n\cos(h\vartheta_n)\\
&\qquad\qquad+h\Omega_2^n\sin(h\vartheta_n)-h\vartheta_n(h\Omega_1^n\sin(h\vartheta_n)+h\Omega_2^n\cos(h\vartheta_n))-h\Omega_2^n],\\
\epsilon^{2}_{(n,4)}=&[-\sin(h\vartheta_n)+h\vartheta_n\cos(h\vartheta_n),\,-\cos(h\vartheta_n)-h\vartheta_n\sin(h\vartheta_n),\,h\Omega_1^n\sin(h\vartheta_n)\\
&\qquad\qquad+h\Omega_2^n\cos(h\vartheta_n)+h\vartheta_n(h\Omega_2^n\sin(h\vartheta_n)-h\Omega_1^n\cos(h\vartheta_n))+h\Omega_1^n],\\
\epsilon^{3}_{(n,4)}=&\epsilon^{3}_{(n-1,4)}=[0,0,1],\\
\epsilon^{1}_{(n,5)}=&[\cos(h\vartheta_n),\,-\sin(h\vartheta_n),\,-h\Omega_1^n\cos(h\vartheta_n)+h\Omega_2^n\sin(h\vartheta_n)],\\
\epsilon^{2}_{(n,5)}=&[\sin(h\vartheta_n),\,\cos(h\vartheta_n),\,-h\Omega_1^n\sin(h\vartheta_n)-h\Omega_2^n\cos(h\vartheta_n)],\\
\epsilon^{1}_{(n-1,4)}=&[\cos(h\vartheta_{n-1}),\,-\sin(h\vartheta_{n-1}),\,-h\Omega_1^{n-1}\cos(h\vartheta_{n-1})+h\Omega_2^{n-1}\sin(h\vartheta_{n-1})],\\
\epsilon^{2}_{(n-1,4)}=&[\sin(h\vartheta_{n-1}),\,\cos(h\vartheta_{n-1}),\,-h\Omega_1^{n-1}\sin(h\vartheta_{n-1})-h\Omega_2^{n-1}\cos(h\vartheta_{n-1})],\\
\epsilon^{\alpha}_{(n-1,5)}=&[0,\,0,\,0],\quad\alpha=1,2,3.
\end{align*}{where we have used that, $\vartheta_n=\Omega_3^n-(\frac{I_2}{I_1+I_2})\frac{\Delta\psi_n}{h}$,  $$\tilde{g}_{n}^{-1}=A_n^{-1}W_n^{-1}=\left[\begin{array}{cc}
   R_{-h\vartheta_n} & -R_{-h\vartheta_n}hv_n \\
   0 & 1 \\
  \end{array}\right],$$ $v_n=[\Omega_1^n,\Omega_2^n]^{T}$ and $R_{h\vartheta_n}=\left[\begin{array}{cc}
   \cos(h\vartheta_n) & \sin(h\vartheta_n) \\
   -\sin(h\vartheta_n) & \cos(h\vartheta_n) \\
  \end{array}\right]$.

Equations \eqref{equation1example}-\eqref{constraintsexample} are used to update the current state $(\tilde a_{n-1}, \tilde a_{n-2}, \lambda^{n-1}, \lambda^{n-2})$  to obtain the next state $(\tilde a_{n}, \tilde a_{n-1}, \lambda^{n}, \lambda^{n-1})$. This is accomplished by solving the dynamics \eqref{equation1example}-\eqref{constraintsexample} with boundary conditions satisfying the constraints \eqref{constraintsexample} using a root-finding algorithm such as Newton's method in terms of the unknowns $(\tilde a_{n}, \lambda^{n})$ to obtain the next configuration.
Note that as in the example in Section \ref{ParticleEx}, the  discrete constrained second-order Lagrange-Poincar\'e equations \eqref{equation1example}- \eqref{constraintsexample} applied to \eqref{AugmSecOrdeDiscLag} represent a discretization in finite differences of \eqref{BeanieCont} (a)-(e). 

It would be interesting to study how indirect optimization methods for optimal control can be developed to implement the equations of motion. One of the main
challenges here is the use of a shooting method to solve the two points boundary
value problem. We believe that, due to the complexity of the equations, one should use multiple shooting instead of a single one, and/or add a final state to our cost functional in order to achieve the desired final configurations. In terms of the integration schemes for Lagrange-Poincare equations, it would be interesting if the variational integrators are employed on particular examples. For instance, in the case of an electron in the magnetic field, the potential function can be used to partially break the symmetry. Then, the extension of the variational integrators presented in this work for symmetry-breaking Lagrange-Poincare systems can be studied independently and applied to a concrete example of interest in physics.

Note also that a particular construction and study for the exact discrete Lagrangian associated to higher-order systems on principal bundles deserve attention and can be an interesting topic to study, based on the previous results obtained in  \cite{CoFeMdD}.

We intend in a future work to explore the role of high-order integrators  \cite{Cedric1}, \cite{Cedric2}, \cite{Sina} iin this class of constrained variational problems for optimal control. As we commented in Section  \ref{DMVI}, higher-order interpolations of the continuous curves lead to more accurate approximations of the exact discrete Lagrangian, and therefore to high-order numerical methods (where here, high-order refers to the local truncation error). This problem, in the context of principal bundles and integration of Lagrange-Poincare equations, is a promising line of investigation, in particular how to relate higher-order constrained variational problems on principal bundles with higher-order integrators, such as Galerkin variational integrators and modified symplectic Runge-Kutta methods, using the results for first-order systems given in  \cite{Cedric2} and \cite{Sina2}.


\end{document}